\documentclass[11pt]{article}
\usepackage{amsmath}
\usepackage{amssymb,amsmath,times,color}
\usepackage{color}
\usepackage{graphicx,epsfig}
\usepackage{amsbsy}
\usepackage{enumitem,algorithm2e,algorithmic}

\usepackage{subfig} 
\usepackage{verbatim}
\usepackage{cite}
\usepackage{subcaption}

\usepackage{geometry}
\usepackage{pict2e}

\usepackage{caption} 
\usepackage{epstopdf}
\usepackage{hyperref}
\usepackage{ulem}

\numberwithin{equation}{section}
\title{High-Resolution Inertial Dynamics with Time-Rescaled Gradients for Nonsmooth Convex Optimization}

\DeclareMathOperator{\prox}{prox}

\newcommand{\interior}{{\rm int}\kern 0.06em}
\newcommand{\inte}{{\rm int}\kern 0.06em}

    \makeatletter
\def\@fnsymbol#1{\ensuremath{\ifcase#1\or 1\or 2 \or\dagger\or \ddagger\or
   \mathsection\or \mathparagraph\or \|\or **\or \dagger\dagger
   \or \ddagger\ddagger \else\@ctrerr\fi}}
    \makeatother

\newcommand{\sgn}{{\rm sgn}\kern 0.12em}

\newcommand{\bd}{{\rm bd}\kern 0.12em}

\newenvironment{proof}{\bigskip\noindent{\bf Proof.}}{\hfill $\blacksquare$ \par\bigskip\noindent}

\newtheorem{lemma}{Lemma}[section]

\newtheorem{remark}{Remark}[section]
\newtheorem{Assumption}{Assumption}[section]

\newtheorem{theorem}{Theorem}[section]

\begin{document}
\date{}
\author{Manh Hung Le\thanks{Unité de Mathématiques Appliquées, ENSTA, Institut Polytechnique de Paris, 91120 Palaiseau, France. E-mail: manh-hung.le@ensta.fr}
	\and
	Andrea Simonetto\thanks{Unité de Mathématiques Appliquées, ENSTA, Institut Polytechnique de Paris, 91120 Palaiseau, France. E-mail: andrea.simonetto@ensta.fr}
}
\maketitle
 {\sc Abstract.} We study nonsmooth convex minimization through a continuous-time dynamical system that can be seen as a high-resolution ODE of Nesterov Accelerated Gradient (NAG) adapted to the nonsmooth case. We apply a time-varying Moreau envelope smoothing to a proper convex lower semicontinuous objective function and introduce a controlled time-rescaling of the gradient, coupled with a Hessian-driven damping term, leading to our proposed inertial dynamic. We provide a well-posedness result for this dynamical system, and construct a Lyapunov energy function capturing the combined effects of inertia, damping, and smoothing. For an appropriate scaling, the energy dissipates and yields fast decay of the objective function and gradient, stabilization of velocities, and weak convergence of trajectories to minimizers under mild assumptions. Conceptually, the system is a nonsmooth high-resolution model of Nesterov's method that clarifies how time-varying smoothing and time rescaling jointly govern acceleration and stability. We further extend the framework to the setting of maximally monotone operators, for which we propose and analyze a corresponding dynamical system and establish analogous convergence results.  We also present numerical experiments illustrating the effect of the main parameters and comparing the proposed system with several benchmark dynamics.

\bigskip


\bigskip


\section{Introduction}

\subsection{Problem and motivation}

Let $\mathcal{H} $ be a real Hilbert space endowed with the inner product $\langle \cdot, \cdot \rangle$ and the associated norm $\Vert \cdot \Vert$. We study the unconstrained convex (possibly nonsmooth) minimization problem
\begin{align}
\label{original pb}
\min_{x\in\mathcal H} f(x),   
\end{align}
where $f:\mathcal{H}\to\mathbb{R}\cup\{+\infty\}$ is proper, convex, and lower semicontinuous, with nonempty solution set $\text{argmin}_{\mathcal{H}}f \neq \emptyset$. Such nonsmooth convex problems arise frequently in applications (signal processing, statistics, imaging, machine learning) especially due to the presence of sparsity-inducing or low-complexity regularizers. First-order {accelerated} methods, which achieve faster convergence rates by incorporating momentum, are particularly valuable for these problems. In discrete time, the prototypical accelerated method is {Nesterov’s accelerated gradient (NAG)} \cite{Nesterov1983}, which for smooth convex $f$ attains an optimal $O(1/k^2)$ convergence rate in function value after $k$ iterations—significantly faster than the $O(1/k)$ rate of gradient descent or Polyak’s heavy-ball method \cite{Polyak1964, Goujaud2025}. Our aim is to study an {inertial} (second-order) {continuous-time} approach to solving \eqref{original pb}, in the spirit of viewing optimization algorithms through the lens of dynamical systems. 

Over the past decade, continuous-time modeling of accelerated methods has provided substantial insights. In particular, \textit{low-resolution} ordinary differential equations (ODEs) have been derived that capture the leading-order behavior of momentum-based algorithms. A landmark result by Su, Boyd, and Candes \cite{Su2016} showed that Nesterov’s method can be seen as the discretization of the second-order ODE
\[
\ddot x(t) + \frac{\alpha}{t}\dot x(t) + \nabla f(x(t)) = 0,
\]
with $\alpha=3$. This vanishing-viscosity ODE is a continuous-time analog of NAG and achieves a convergence rate $f(x(t)) - \min f = O(1/t^2)$ for $\alpha \ge 3$. Subsequent investigations \cite{ACPR,AP_smallo} established that by using a slightly larger damping coefficient (effectively $\alpha > 3$ in the above ODE), one can ensure the weak convergence of $x(t)$ to a minimizer and even improve the decay rate to $o(1/t^2)$. This line of work forged a firm link between Nesterov’s discrete method and Lyapunov-based convergence analysis in continuous time. A complementary perspective was provided by Wibisono, Wilson, and Jordan \cite{Wibisono2016}, who formulated a Bregman Lagrangian whose Euler–Lagrange equation recovers a family of accelerated flows. Their variational framework explains many accelerated algorithms (including NAG and its non-Euclidean counterparts) as different time-parameterizations of a single underlying trajectory \cite{Wibisono2016}. 
\subsection{High-resolution ODE models} While the classical ODE $$\ddot{x}(t) + \frac{3}{t}\dot{x}(t) + \nabla f(x) = 0$$ replicates NAG’s $O(1/t^2)$ rate, it does so only at {leading order} and neglects certain subtle effects of discretization. Researchers have therefore sought {higher-resolution} differential equations that more faithfully reflect the behavior of the actual algorithms. A major work in this direction was the introduction of a {Hessian-driven damping} term by Attouch, Peypouquet, and Redont \cite{AttouchPeyp2016}. They augmented the Nesterov ODE with an additional term involving the Hessian $\nabla^2 f$, obtaining the flow 
\[
\ddot x(t)+\frac{\alpha}{t}\dot x(t)+\beta\,\nabla^2 f(x(t))\,\dot x(t)+\,\nabla f(x(t))=0,
\]
with $\alpha \ge 3$ and $\beta>0$. This enriched model retains the $O(1/t^2)$ convergence rate in function value, while dramatically reducing oscillations and guaranteeing a rapid decay of $\Vert\nabla f(x(t))\Vert$ to $0$. The Hessian-driven damping term can be interpreted geometrically as a kind of curvature-dependent friction: since $\nabla^2 f(x)\dot{x} = \frac{\mathrm{d}}{\mathrm{d}t}\nabla f(x(t))$, the extra term acts like a derivative of the gradient and helps to attenuate the oscillatory energy of the system. Building on this idea, Attouch, Chbani, Fadili, and Riahi performed a detailed study of a more general version of the above dynamics \cite{AttouchChbani2022} as
\[
\mathrm{(DIN\!-\!AVD)}_{\alpha,\beta,b(t)}\qquad
\ddot x(t)+\frac{\alpha}{t}\dot x(t)+\beta\,\nabla^2 f(x(t))\,\dot x(t)+b(t)\,\nabla f(x(t))=0.
\]
They demonstrated via a Lyapunov analysis that the continuous trajectory enjoys fast convergence of both function values and gradients. A particularly convenient and theoretically revealing choice is \(b(t)=1+\beta/t\). It is worth noticing that in \cite{AttouchChbani2022} the authors  performed a principled discretization of their second-order dynamical system \(\mathrm{(DIN\!-\!AVD)}_{\alpha,\beta,1+\beta/t}\) and obtained an algorithm named {IGAHD}; the numerical experiments carried out by the authors on \(\ell_1\), \(\ell_1\!-\!\ell_2\), TV, and nuclear-norm models showed that IGAHD reduced oscillations and converge at a faster rate than FISTA (the \(\beta=0\) case). In a subsequent work, Attouch and Fadili \cite{AttouchFadili2022} revisited the heavy-ball (Polyak) and Nesterov methods from a dynamical systems viewpoint. They established the so-called high-resolution ODEs corresponding to both methods and revealed that {both} involve a Hessian-driven geometric damping term in continuous time. Specifically, they obtained the following high-resolution dynamic of NAG:
\[
\ddot x(t)\;+\;{\sqrt{s}\,\nabla^2 f(x(t))\,\dot x(t)}\;+\;\frac{\alpha}{t}\dot x(t)\;+\;\Big(1+\frac{\alpha\sqrt{s}}{2t}\Big)\nabla f(x(t))\;=\;0.
\]
Their high-resolution approach reveals two additional structures in the continuous-time limit of NAG:
\begin{enumerate}
    \item A {Hessian-driven damping} term \(\nabla^2 f(x)\,\dot x\), and
    \item A {time-rescaled gradient} of the form \(\big(1+\tfrac{\mathrm{const}}{t}\big)\nabla f(\cdot)\).
\end{enumerate}
Interestingly, we see that the appearance of Hessian-driven damping can be rigorously understood as a consequence of the high-resolution ODE framework. 

Further highlighting the importance of high-resolution modeling, Shi, Du, Su, and Jordan \cite{Shi2019} pointed out that the classical low-resolution ODEs (such as the Su–Boyd–Candes dynamic) do not actually distinguish between Polyak’s momentum and Nesterov’s method in certain regimes. By deriving higher-order differential equations and applying a symplectic (structure-preserving) integrator, Shi et al. were able to recover true accelerated algorithms in discrete time for smooth strongly convex problems, whereas naive Euler discretizations of the basic ODE often fail to do so \cite{Shi2019}. All of these works underscore that incorporating second-order information (e.g. Hessian damping) and time-rescaled gradients into the continuous model is crucial for capturing and explaining the distinct behavior of accelerated optimization methods.
\subsection{Time-rescaled dynamics}
Another fruitful direction in continuous optimization dynamics is the use of {time rescaling} to accelerate convergence. In standard gradient flow or Nesterov’s flow, time is uniform; but one can intentionally speed up the system’s evolution by rescaling time (or equivalently, introducing time-varying coefficients in the ODE). For example, consider modifying the Nesterov ODE to
\[
\ddot{x}(t)+\frac{\alpha}{t}\dot{x}(t)+b(t)\,\nabla f(x(t))=0,
\]
where $b$ is an increasing function. As shown by Attouch, Chbani, and Riahi \cite{AttouchRiahi2019}, the factor $b(t)$ effectively accelerates the driving gradient force and can yield improved decay rates of order $O\big(\frac{1}{t^2b(t)}\big)$.
\subsection{Moreau envelope smoothing for nonsmooth problems}
The discussion so far has assumed $f$ is smooth (at least $C^1$ with Lipschitz gradient). However, our ultimate interest is the nonsmooth case, where $f$ may not be differentiable. A standard approach to handle a nonsmooth convex $f$ is to introduce its {Moreau envelope}. For \(\gamma>0\), the {Moreau envelope} of a proper, convex, and lower semicontinuous function 
$f:\mathcal{H}\to\mathbb{R}\cup\{+\infty\}$
 with parameter \(\gamma\) is defined by
\[
f_\gamma(x) := \min_{y \in \mathcal H} \Bigl\{ f(y) + \tfrac{1}{2\gamma}\|x - y\|^2 \Bigr\}.
\]
The unique minimizer in the definition is given by the {proximal map}
\[
\operatorname{prox}_{\gamma f}(x) := \arg\min_{y \in \mathcal H} \Bigl\{ f(y) + \tfrac{1}{2\gamma}\|x - y\|^2 \Bigr\}.
\]
The Moreau envelope satisfies the following classical properties:
\begin{enumerate}
    \item \(f_\gamma\) is convex and continuously differentiable even if \(f\) is nonsmooth;
    \item The gradient is Lipschitz with
    \[
    \nabla f_\gamma(x) = \tfrac{1}{\gamma}\bigl(x - \operatorname{prox}_{\gamma f}(x)\bigr)
    \qquad\text{and}\qquad
    \|\nabla f_\gamma(x) - \nabla f_\gamma(y)\| \le \tfrac{1}{\gamma}\|x-y\|;
    \]
    \item The envelope preserves the set of minimizers, i.e.,
    \[
    \arg\min f_\gamma = \arg\min f.
    \]
\end{enumerate}
Hence, \(f_\gamma\) provides a smooth approximation of \(f\) that retains the same minimizers and has a controlled gradient Lipschitz constant \(1/\gamma\). 

This smoothing technique has been leveraged in continuous dynamics by Attouch and Cabot \cite{AttouchCabot2018JDE}, who were among the first to analyze an {inertial Moreau-envelope flow} for nonsmooth optimization. They considered the ODE
\begin{align*}
    \ddot{x}(t)+\frac{\alpha}{t}\dot{x}(t)+\nabla f_{\lambda(t)}(x(t))=0,
\end{align*}
where $\lambda(t)$ is a smoothing parameter that may slowly vary with time. They proved that by appropriately choosing $\lambda(t)$, the trajectory $x(t)$ enjoys accelerated convergence in terms of the smoothed objective: $f_{\lambda(t)}(x(t)) - f_* = o(1/t^2)$, along with $\Vert \dot{x}(t)\Vert = o(1/t)$ and weak convergence of $x(t)$ to a minimizer. Furthermore, these convergence rates for the envelope translate back to the original objective via a {proximal shadow} argument, meaning the true function values $f(x(t))$ also approach the optimum. More recently, Attouch and Laszló \cite{AttouchLaszlo2021} incorporated Hessian-driven damping into such an envelope-based inertial dynamic. They showed that adding the $\beta\frac{\mathrm{d}}{\mathrm{d}t}\nabla f_{\lambda}(x(t))$ term yields similar convergence properties as in the smooth case (fast decay of the envelope’s gradient norm, etc.) Concurrently, other researchers have explored continuous-time methods that combine envelope smoothing with time rescaling or other regularization techniques. For example, Boţ and Karapetyants \cite{BotKarapetyants2022} proposed a fast inertial flow for nonsmooth convex problems that employs an explicit time scaling together with the Moreau envelope smoothing to achieve accelerated convergence. In a follow-up work, Karapetyants \cite{Karapetyants2023} studied a related continuous-time approach using Tikhonov (ridge) regularization of the objective, and provided convergence guarantees for the resulting trajectories. Other relevant literature is captured by the works of ~\cite{Adly2026,Cortild2024,AttouchLaszlo2021, AttouchSvaiter2011,WangFadiliOchs2025,AAF}. In general, these works confirm that the powerful acceleration principles known for smooth dynamics (momentum, time rescaling, etc.) can be successfully transferred to the nonsmooth setting by working with regularized/smoothed surrogates of $f$.

\subsection{Our dynamics: a nonsmooth high-resolution approach with time-rescaled gradients}

We build on this rich literature by formulating a new high-resolution inertial dynamic tailored to {nonsmooth} convex minimization. Our dynamical system employs both the high-resolution framework \cite{AttouchFadili2022} and the {time-rescaled gradients} (as in Attouch–Chbani–Riahi \cite{AttouchRiahi2019}), applied in conjunction with a time-varying Moreau envelope smoothing of $f$.
\begin{equation}
\boxed{\;
\ddot x(t)\;+\;\frac{\alpha}{t}\dot x(t)\;+\;\beta\,\frac{\mathrm{d}}{\mathrm{d}t}\!\Big[\delta(t)\,\nabla f_{\gamma(t)}(x(t))\Big]
\;+\;\Big(1+\frac{\beta}{t}\Big)\delta(t)\,\nabla f_{\gamma(t)}(x(t))\;=\;0.
\;}
\quad (\text{NS-HR})_{\alpha,\beta}
\end{equation}
Here the parameters have the following roles:
\begin{itemize}
    \item \( \alpha>0 \): controls the {vanishing viscous damping} coefficient \( \alpha/t \);
    \item \( \beta>0 \): weights the { Hessian-driven damping} contribution;
    \item \( \gamma(t)>0 \): specifies the {time-varying smoothing parameter} in the Moreau envelope \( f_{\gamma(t)} \);
    \item \( \delta(t)>0 \): governs the {time-rescaling of the gradient}.
\end{itemize} 
Let us stress that the envelope yields \(C^1\) smoothness with Lipschitz gradient, preserves \(\arg\min f\), and lets us encode {Hessian-driven damping} through the time derivative of \(\nabla f_{\gamma(t)}\) without computing Hessians.
It is worth noticing that our dynamics can be seen as a nonsmooth, time-rescaled extension of the $\mathrm{(DIN\!-\!AVD)}_{\alpha,\beta,1+\beta/t}$ studied in the smooth case by Attouch et al. \cite{AttouchChbani2022}.\\[2mm]
We will also extend $(\text{NS-HR})_{\alpha,\beta}$ to deal with the problem of finding zeros of maximally monotone operators. Specifically, we are interested in: 
\begin{align*}
    \text{Find } x \in \mathcal{H}: 0 \in A(x),
\end{align*}
where $A:\mathcal{H}\to 2^{\mathbf{\mathcal{H}}}$ is a maximally monotone operator. \\
Analogously to the case of nonsmooth optimization, we propose the following dynamic:  
\begin{equation*}
\boxed{\ddot x(t)\;+\;\frac{\alpha}{t}\dot x(t)\;+\;\beta\,\frac{\mathrm{d}}{\mathrm{d}t}\!\Big[\delta(t)\,A_{\gamma(t)}(x(t))\Big]
\;+\;\Big(1+\frac{\beta}{t}\Big)\delta(t)\,A_{\gamma(t)}(x(t))\;=\;0,} \quad (\text{HR-MMD})_{\alpha,\beta}
\end{equation*}
where $A_{\gamma(t)}$ is the Yosida approximation of $A$ with constant $\gamma(t)$. \\[3mm]
\textbf{Particularizations and links to existing dynamics.}
The proposed dynamical system $(\mathrm{NS\text{-}HR})_{\alpha,\beta}$ can be viewed as a unifying template that
recovers several classical continuous-time models for accelerated optimization as soon as the parameters
$(\alpha,\beta,\delta,\gamma)$ are specialized.

\smallskip
\noindent
\emph{Low-resolution and time-rescaled limits.}
Setting $\beta=0$ removes the high-resolution term and yields a low-resolution inertial envelope flow \cite{BotKarapetyants2022}
\[
\ddot x(t)+\frac{\alpha}{t}\dot x(t)+\delta(t)\nabla f_{\gamma(t)}(x(t))=0.
\]
When $\delta(t)\equiv 1$, this recovers the nonsmooth inertial Moreau-envelope dynamics studied in
\cite{AttouchCabot2018JDE} (with time-dependent smoothing). When $\gamma(t)$ is neglected, one obtains a time-rescaled Nesterov-type flow of the form
$\ddot x+\frac{\alpha}{t}\dot x + b(t)\nabla f(x)=0$ \cite{AttouchRiahi2019}. In particular, taking $\delta(t)\equiv 1$ and $\beta=0$ recovers the Su--Boyd--Cand\`es ODE
$\ddot x + \frac{\alpha}{t}\dot x + \nabla f(x)=0$ in the smooth setting \cite{Su2016}.

\smallskip
\noindent
\emph{High-resolution/Hessian-driven damping limits.}
When $\beta>0$, the term $\frac{d}{dt}(\delta(t)\nabla f_{\gamma(t)}(x(t)))$ introduces a Hessian-driven damping
effect without explicit Hessian computations, since for smooth $f$ one has
$\frac{d}{dt}(\nabla f(x(t)))=\nabla^2 f(x(t))\dot x(t)$.
If $f$ is smooth, $\delta(t)\equiv 1$, and $\gamma(t)$ is neglected, $(\mathrm{NS\text{-}HR})_{\alpha,\beta}$ reduces to the
high-resolution inertial system $(\mathrm{DIN\text{-}AVD})_{\alpha,\beta,\,1+\beta/t}$ studied in \cite{AttouchChbani2022},
which underlies the IGAHD discretization (in the same paper). Overall, $(\mathrm{NS\text{-}HR})_{\alpha,\beta}$ provides a common envelope-based extension of low-resolution Nesterov flows,
time-rescaled inertial dynamics, and high-resolution Hessian-damped models.

\begin{table}[t]
\centering
\caption{Particularizations of $(\mathrm{NS\text{-}HR})_{\alpha,\beta}$ highlighted in the discussion.}
\label{tab:particularizations_discussion}
\renewcommand{\arraystretch}{1.2}\centering
\begin{tabular}{p{0.30\linewidth} p{0.48\linewidth} c}
\hline
\textbf{Parameter choice} & \textbf{Resulting dynamics} & \textbf{Reference} \\
\hline

$\beta=0$ (general $\delta,\gamma$) 
&
$\ddot x(t)+\frac{\alpha}{t}\dot x(t)+\delta(t)\nabla f_{\gamma(t)}(x(t))=0$
(low-resolution inertial envelope flow)
& \cite{BotKarapetyants2022} \\\hline

$\beta=0,\ \delta(t)\equiv 1$ 
&
$\ddot x(t)+\frac{\alpha}{t}\dot x(t)+\nabla f_{\gamma(t)}(x(t))=0$
(nonsmooth inertial Moreau-envelope dynamics with smoothing)
& \cite{AttouchCabot2018JDE} \\ \hline

$\beta=0,\ \gamma(t)\ \text{neglected}$, \newline $\delta(t)=b(t)$
&
$\ddot x(t)+\frac{\alpha}{t}\dot x(t)+b(t)\nabla f(x(t))=0$
(time-rescaled Nesterov-type flow)
& \cite{AttouchRiahi2019} \\\hline

$\beta=0,\ \delta(t)\equiv 1$, \newline $ \gamma(t)\ \text{neglected}$
&
$\ddot x(t)+\frac{\alpha}{t}\dot x(t)+\nabla f(x(t))=0$
(Su--Boyd--Cand\`es low-resolution ODE)
& \cite{Su2016} \\\hline

$\beta>0,\ \delta(t)\equiv 1$, \newline $ \gamma(t)\ \text{neglected},\ f\ \text{smooth}$
&
$(\mathrm{DIN\text{-}AVD})_{\alpha,\beta,\,1+\beta/t}$
(high-resolution inertial system / Hessian-driven damping limit)
& \cite{AttouchChbani2022} \\

\hline
\end{tabular}
\end{table}

\textbf{Contribution of the paper} 
\begin{enumerate}
    \item \textbf{Nonsmooth high-resolution dynamic with time-rescaled gradients.} We propose and study the dynamical system $(\text{NS-HR})_{\alpha,\beta}$, which can be viewed as a high-resolution continuous-time model tailored for {nonsmooth} convex optimization. This model incorporates several key mechanisms simultaneously: a vanishing viscous damping term that governs long-term stabilization, a {Hessian-driven damping} component inherited from the high-resolution ODE perspective, and a gradient term whose magnitude is appropriately rescaled in time. We further show that $(\text{NS-HR})_{\alpha,\beta}$ admits an equivalent first-order formulation, which plays a crucial role in potential discretizations. In addition, we establish a rigorous well-posedness theory for the system by proving global existence and uniqueness of solutions under natural assumptions on the data.
    \item \textbf{Lyapunov framework and asymptotics.} We construct a Lyapunov function specifically adapted to $(\text{NS-HR})_{\alpha,\beta}$. This functional allows us to derive a detailed quantitative description of the asymptotic behavior of the trajectories. In particular, we obtain arbitrarily fast decay rates for both objective function values and gradient norms, demonstrating that the system can achieve accelerated convergence in a precise sense. Furthermore, we show stabilization of the velocity field and prove the {weak convergence} of trajectories toward minimizers of the underlying convex function, thereby providing a complete qualitative picture of the dynamics.
    \item \textbf{Extension to maximally monotone operators.}  
We propose the dynamical system $(\text{HR-MMD})_{\alpha,\beta}$ and extend the analysis developed in the convex optimization setting to the broader and more abstract framework of finding zeros of maximally monotone operators. In this extension, we demonstrate that the fundamental structural results, including well-posedness, Lyapunov decay estimates, and convergence properties, carry over under appropriate assumptions.
\item \textbf{Numerical experiments and a proximal discretization perspective.}
We complement the theoretical analysis with a series of numerical experiments illustrating the influence of the parameters \(\beta\) and \(\alpha\), as well as comparisons with several benchmark dynamics from the literature. These experiments highlight the damping effect of the Hessian-driven term, confirm our theoretical results, and show the favorable behavior of the proposed system relative to competing approaches. In addition, we derive a natural second-order time discretization of \((\mathrm{NS\text{-}HR})_{\alpha,\beta}\) and show that the resulting implicit scheme can be resolved by a single proximal step, thereby providing an algorithmic counterpart of the continuous model.
\end{enumerate}

\subsection{Paper outline}
The remainder of the paper is organized as follows. Section~2 derives an equivalent first-order reformulation of \((\mathrm{NS\text{-}HR})_{\alpha,\beta}\), and establishes global existence and uniqueness of solution trajectories. Section~3 develops the main Lyapunov framework and provides the asymptotic analysis of the system, including fast decay rates for the Moreau-envelope objective residual and gradient norm, stabilization of velocities, and weak convergence of trajectories under suitable assumptions on the time-dependent parameters. Section~4 extends the analysis to the setting of maximally monotone operators through the corresponding high-resolution Yosida-regularized dynamic. Section~5 discusses the relationship between the two analyses and highlights similarities and differences between the optimization and monotone-inclusion settings. Section~6 discusses a natural discretization of \((\mathrm{NS\text{-}HR})_{\alpha,\beta}\), showing that the resulting implicit scheme can be resolved through a single proximal step. Section~7 is devoted to numerical experiments: we investigate the influence of the parameters \(\beta\) and \(\alpha\), and compare the proposed system with several reference dynamics from the literature.  Finally, Section~8 concludes the paper and outlines several perspectives for future work.
\section{Wellposedness of $(\text{NS-HR})_{\alpha,\beta}$}

\subsection{First-order reformulation of $(\text{NS-HR})_{\alpha,\beta}$}

We start with a first-order reformulation of $(\text{NS-HR})_{\alpha,\beta}$, which we will leverage in the subsequent proofs. 
\begin{theorem}
\label{thm:first-order}
    {Suppose $\delta:[t_0,+\infty
    ) \to (0,+\infty)$ and $\gamma:[t_0,+\infty
    ) \to (0,+\infty)$ are continuously differentiable functions}. Let $(x_0,v_0,y_0) \in \mathcal{H}^3$ satisfying 
    \begin{align}
    \label{compatibility condition}
        v_0=-\beta\delta(t_0)\nabla f_{\gamma(t_0)}(x_0)-\Big(\frac{\alpha-1}{t_0}-\frac{1}{\beta}\Big)x_0-\frac{1}{\beta}y_0.
    \end{align}
    Then, the following are equivalent 
    \begin{enumerate}[label=(\roman*)]
        \item $x$ is a solution to $(\text{NS-HR})_{\alpha,\beta}$ with $x(t_0)=x_0$, and $\dot{x}(t_0)=v_0$. 
        \item There exists $y$ such that $(x,y)$ satisfies 
\begin{align}
 \dot{x}(t)&+\beta\delta(t)\nabla f_{\gamma(t)}(x(t))+\Big(\frac{\alpha-1}{t}-\frac{1}{\beta}\Big)x(t)+\frac{1}{\beta}y(t)=0, \label{eq:x}\\
\dot{y}(t)
&- \Big(\frac{1}{\beta}-\frac{\alpha-2}{t}\Big)x(t)
+ \Big(\frac{1}{t}+\frac{1}{\beta}\Big)y(t)
= 0, \label{eq:y}
\end{align}
        with $(x(t_0),y(t_0))=(x_0,y_0)$. 
    \end{enumerate}
\end{theorem}


\begin{proof}
\noindent\textbf{(ii)$\;\Rightarrow\;$ (i).}
Let us differentiate \eqref{eq:x}:
\[
\ddot{x}(t)+\beta\frac{\mathrm{d}}{\mathrm{d}t}\Big[\delta(t)\nabla f_{\gamma(t)}(x(t))\Big]+\Big(\frac{\alpha-1}{t}-\frac{1}{\beta}\Big)\dot{x}(t)-\frac{\alpha-1}{t^2}x(t)+\frac{1}{\beta}\dot{y}(t)=0.
\]
This implies 
\begin{align}
\label{derivative of first eq}
    \ddot{x}(t)+\frac{\alpha}{t}\dot{x}(t)+\beta\frac{\mathrm{d}}{\mathrm{d}t}\Big[\delta(t)\nabla f_{\gamma(t)}(x(t))\Big]-\Big(\frac{1}{t}+\frac{1}{\beta}\Big)\dot{x}(t)-\frac{\alpha-1}{t^2}x(t)+\frac{1}{\beta}\dot{y}(t)=0.
\end{align}
From \eqref{derivative of first eq}, to conclude that $x$ has to satisfy $(\text{NS-HR})_{\alpha,\beta}$, it suffices to prove that
\begin{align*}
-\Big(\frac{1}{t}+\frac{1}{\beta}\Big)\dot{x}(t)-\frac{\alpha-1}{t^2}x(t)+\frac{1}{\beta}\dot{y}(t)=\Big(1+\frac{\beta}{t}\Big)\delta(t)\nabla f_{\gamma(t)}(x(t)). 
\end{align*}
Indeed, from the first-order formulation we yield 
\begin{align*}
&-\Big(\frac{1}{t}+\frac{1}{\beta}\Big)\dot{x}(t)-\frac{\alpha-1}{t^2}x(t)+\frac{1}{\beta}\dot{y}(t) \\
=&-\Big(\frac{1}{t}+\frac{1}{\beta}\Big)\Big[-\beta\delta(t)\nabla f_{\gamma(t)}(x(t))-\Big(\frac{\alpha-1}{t}-\frac{1}{\beta}\Big)x(t)-\frac{1}{\beta}y(t)\Big]-\frac{\alpha-1}{t^2}x(t) \\
&+\frac{1}{\beta}\Big[\Big(\frac{1}{\beta}-\frac{\alpha-2}{t}\Big)x(t)-\Big(\frac{1}{t}+\frac{1}{\beta}\Big)y(t)\Big] \\
=&\Big(1+\frac{\beta}{t}\Big)\delta(t)\nabla f_{\gamma(t)}(x(t)) +\Big[\Big(\frac{1}{t}+\frac{1}{\beta}\Big)\Big(\frac{\alpha-1}{t}-\frac{1}{\beta}\Big)-\frac{\alpha-1}{t^2}+\frac{1}{\beta}\Big(\frac{1}{\beta}-\frac{\alpha-2}{t}\Big)\Big]x(t)\\
&+\Big[\Big(\frac{1}{t}+\frac{1}{\beta}\Big)\Big(\frac{1}{\beta}\Big)-\frac{1}{\beta}\Big(\frac{1}{t}+\frac{1}{\beta}\Big)\Big]y(t). 
\end{align*}
The proof for this implication is thereby completed by noticing 
$$
\begin{cases}
    \Big(\frac{1}{t}+\frac{1}{\beta}\Big)\Big(\frac{\alpha-1}{t}-\frac{1}{\beta}\Big)-\frac{\alpha-1}{t^2}+\frac{1}{\beta}\Big(\frac{1}{\beta}-\frac{\alpha-2}{t}\Big)=0, \\
    \Big(\frac{1}{t}+\frac{1}{\beta}\Big)\Big(\frac{1}{\beta}\Big)-\frac{1}{\beta}\Big(\frac{1}{t}+\frac{1}{\beta}\Big)=0,
\end{cases}
$$
and that the initial condition is automatically satisfied by the requirement on $v_0$ expressed in Eq.~\eqref{compatibility condition}. 

\noindent\textbf{(i)$\;\Rightarrow\;$ (ii).}
Define $y:[t_0,+\infty) \longrightarrow \mathcal{H}$ according to the requirement on $v_0$ expressed in Eq.~\eqref{compatibility condition} as well as the equation 
\begin{align}
\label{define y}
    \dot{x}(t)+\beta\delta(t)\nabla f_{\gamma(t)}(x(t))+\Big(\frac{\alpha-1}{t}-\frac{1}{\beta}\Big)x(t)+\frac{1}{\beta}y(t)=0.
\end{align}
As previously seen, the differentiation of this equation with respect to $t$ gives the second-order dynamic
\begin{align}
\label{first diff}
    \ddot{x}(t)+\frac{\alpha}{t}\dot{x}(t)+\beta\frac{\mathrm{d}}{\mathrm{d}t}\Big[\delta(t)\nabla f_{\gamma(t)}(x(t))\Big]-\Big(\frac{1}{t}+\frac{1}{\beta}\Big)\dot{x}(t)-\frac{\alpha-1}{t^2}x(t)+\frac{1}{\beta}\dot{y}(t)=0.
\end{align}
From the definition of $x$ and $y$, namely $(\text{NS-HR})_{\alpha,\beta}$ and \eqref{define y} we have 
$$
\begin{cases}
    \ddot{x}(t)+\frac{\alpha}{t}\dot{x}(t)+\beta\frac{\mathrm{d}}{\mathrm{d}t}\Big[\delta(t)\nabla f_{\gamma(t)}(x(t))\Big]=-\Big(1+\frac{\beta}{t}\Big)\delta(t)\nabla f_{\gamma(t)}(x(t)),\\
    -\Big(\frac{1}{t}+\frac{1}{\beta}\Big)\dot{x}(t)=\Big(1+\frac{\beta}{t}\Big)\delta(t)\nabla f_{\gamma(t)}(x(t))+\Big(\frac{1}{t}+\frac{1}{\beta}\Big)\Big(\frac{\alpha-1}{t}-\frac{1}{\beta}\Big)x(t)+\frac{1}{\beta}\Big(\frac{1}{t}+\frac{1}{\beta}\Big)y(t).
\end{cases}
$$
Plugging these identities in \eqref{first diff} gives
\begin{align*}
-\Big(1+\frac{\beta}{t}\Big)\delta(t)\nabla f_{\gamma(t)}(x(t))+\Big(1+\frac{\beta}{t}\Big)\delta(t)\nabla f_{\gamma(t)}(x(t))+\Big(\frac{1}{t}+\frac{1}{\beta}\Big)\Big(\frac{\alpha-1}{t}-\frac{1}{\beta}\Big)x(t)\\+\frac{1}{\beta}\Big(\frac{1}{t}+\frac{1}{\beta}\Big)y(t)-\frac{\alpha-1}{t^2}x(t)+\frac{1}{\beta}\dot{y}(t)=0.
\end{align*}
Simplification yields
\begin{align*}
 \frac{1}{\beta}\Big(\frac{\alpha-2}{t}-\frac{1}{\beta}\Big)x(t)+\frac{1}{\beta}\Big(\frac{1}{t}+\frac{1}{\beta}\Big)y(t)+\frac{1}{\beta}\dot{y}(t)=0.
\end{align*}
Multiplying this equation by $\beta$ gives \eqref{eq:y}, hence completing the proof. 
\end{proof}
\subsection{Existence and uniqueness of a solution trajectory of $(\text{NS-HR})_{\alpha,\beta}$}

To establish the existence and uniqueness of a solution to $(\text{NS-HR})_{\alpha,\beta}$, we make use of its first-order reformulation. Set $z(t)=(x(t),y(t))$, and $F:[t_0,+\infty)\times\mathcal{H}^2\to\mathcal{H}^2$ defined by 
\begin{align}\label{1stds}
    F(t,(x,y))=\Big[-\beta\delta(t)\nabla f_{\gamma(t)}(x)-\Big(\frac{\alpha-1}{t}-\frac{1}{\beta}\Big)x-\frac{1}{\beta}y, \Big(\frac{1}{\beta}-\frac{\alpha-2}{t}\Big)x-\Big(\frac{1}{t}+\frac{1}{\beta}\Big)y\Big].
\end{align}
$(\text{NS-HR})_{\alpha,\beta}$ can now be written equivalently as 
\begin{align*}
    \dot{z}(t)=F(t,z(t)).
\end{align*}
We now will rely on the following result \cite[Proposition 6.2.1]{Haraux}.
\begin{theorem} \cite[Proposition 6.2.1]{Haraux}
\label{thm:existence_uniqueness}
    Let $X$ be a Banach space, $s_0>0$, and  $g:[s_0,+\infty)\times X \longrightarrow X$. Suppose that $g$ satisfies the following properties:
    \begin{enumerate}[label=(\roman*)]
        \item $\forall x\in X$, $g(\cdot,x) \in L^1_{loc}([s_0,+\infty),X)$. 
        \item For a.e. $t \in [s_0,+\infty)$, $g(t,\cdot):X\longrightarrow X$ is continuous and satisfies 
        \begin{align*}
            \forall x, y \in X, \Vert g(t,x)-g(t,y) \Vert \le K(t,\Vert x\Vert +\Vert y \Vert)\Vert x-y\Vert,
        \end{align*}
        where $K(\cdot,r) \in L^1_{loc}([s_0,+\infty)) \quad \forall r\geq 0$. 
        \item For a.e. $t \in [s_0,+\infty)$, $g(t,\cdot):X\longrightarrow X$ satisfies 
        \begin{align*}
            \Vert g(t,x) \Vert \le P(t)(1+\Vert x \Vert) \quad \forall x \in X,
        \end{align*}
        where $P \in L^1_{loc}([s_0,+\infty))$.
    \end{enumerate}
    Then, for every $u_0 \in X$, there exists a unique function $u \in W^{1,1}_{loc}([s_0,+\infty),X)$, i.e., $u$ is locally absolutely continuous on $[s_0,+\infty)$, such that 
    $$
    \begin{cases}
        \dot{u}(t)=g(t,u(t)) \text{ for a.e. } t \in [s_0,+\infty), \\
        u(s_0)=u_0.
    \end{cases}
    $$
\end{theorem}
We are now ready to present and prove the existence and uniqueness result for $(\text{NS-HR})_{\alpha,\beta}$.
\begin{theorem}
    Let $t_0>0$ be an arbitrary starting time. Suppose $\delta:[t_0,+\infty
    ) \to (0,+\infty)$ and $\gamma:[t_0,+\infty
    ) \to (0,+\infty)$ are continuously differentiable functions. Then given any pair $(x_0,v_0) \in \mathcal{H}^2$, there exists a unique solution trajectory $x \in W^{1,1}_{loc}([t_0,+\infty))$ that satisfies $(\text{NS-HR})_{\alpha,\beta}$ almost everywhere on $[t_0,+\infty)$ with the intial condition $(x(t_0),\dot{x}(t_0))=(x_0,v_0)$.
\end{theorem}

\begin{proof}
One can see that for all $(x,y) \in \mathcal{H}^2$, $F(\cdot,x,y)$ defined in Eq.~\eqref{1stds} is continuous on $[t_0,+\infty)$, and therefore $F(\cdot,x,y) \in L^1_{loc}([t_0,+\infty),\mathcal{H}^2)$. It remains to verify that $F$ satisfies the last two requirements of Theorem~\ref{thm:existence_uniqueness}. 

\textbf{Verify that $F$ satisfies the requirement \textit{(ii)} of Theorem \ref{thm:existence_uniqueness}} 

Given $(x,y)\in \mathcal{H}^2$ and $(u,v)\in \mathcal{H}^2$, we consider the difference 
\begin{align*}
    \Vert F(t,x,y)-F(t,u,v)\Vert =&\Big[\Big\Vert \beta\delta(t)\Big(\nabla f_{\gamma(t)}(u)-\nabla f_{\gamma(t)}(x)\Big) +\Big(\frac{\alpha-1}{t}-\frac{1}{\beta}\Big)(u-x)+\frac{1}{\beta}(v-y)\Big\Vert^2\\
    &+\Big\Vert \Big(\frac{1}{\beta}-\frac{\alpha-2}{t}\Big)(x-u)+\Big(\frac{1}{t}+\frac{1}{\beta}\Big)(v-y) \Big\Vert^2\Big]^{1/2}
\end{align*}
This equality, combined with the $\frac{1}{\gamma(t)}-$Lipschitz continuity of $\nabla f_{\gamma(t)}$ and the triangle inequality, implies 
\begin{align*}
\Vert F(t,x,y)-F(t,u,v)\Vert &\le \beta\frac{\delta(t)}{\gamma(t)}\Vert u-x \Vert + \Big\vert\frac{\alpha-1}{t}-\frac{1}{\beta}\Big\vert\Vert u-x\Vert+\frac{1}{\beta}\Vert v-y \Vert\\
&\hspace{5mm}+\Big\vert\frac{1}{\beta}-\frac{\alpha-2}{t}\Big\vert\Vert x-u \Vert 
+\Big\vert\frac{1}{t}+\frac{1}{\beta}\Big\vert\Vert v-y \Vert \\
&\le a(t)\Vert u-x \Vert +b(t) \Vert v-y \Vert,
\end{align*}
where we have conveniently set
\begin{align*}
&a(t)=\beta\frac{\delta(t)}{\gamma(t)} +\Big\vert \frac{\alpha-1}{t}-\frac{1}{\beta}\Big\vert+\Big\vert \frac{1}{\beta}-\frac{\alpha-2}{t}\Big\vert,\\
    &b(t)=\frac{1}{\beta}+\Big\vert\frac{1}{t}+\frac{1}{\beta}\Big\vert.
\end{align*}
We further obtain 
\begin{align*}
\Vert F(t,x,y)-F(t,u,v)\Vert \le K(t)\Vert (x,y)-(u,v) \Vert,
\end{align*}
where $K(t):=\Big(a(t)^2+b(t)^2\Big)^{1/2}$. 

As such, since $a \in L^1_{loc}([t_0,+\infty))$ and $b \in L^1_{loc}([t_0,+\infty))$, then $K \in L^1_{loc}([t_0,+\infty))$. 

\textbf{Verify that $F$ satisfies the requirement \textit{(iii)} of Theorem \ref{thm:existence_uniqueness}} 

Consider any minimizer of $f$, which we indicate as $x^*$. From the definition of $F$, we have 
\begin{align*}
    \Vert F(t,x,y) \Vert &\le \Big\Vert -\beta\delta(t)\nabla f_{\gamma(t)}(x)-\Big(\frac{\alpha-1}{t}-\frac{1}{\beta}\Big)x-\frac{1}{\beta}y \Big\Vert 
    +\Big\Vert\Big(\frac{1}{\beta}-\frac{\alpha-2}{t}\Big)x-\Big(\frac{1}{t}+\frac{1}{\beta}\Big)y \Big\Vert \\
    &\le \beta\frac{\delta(t)}{\gamma(t)}\Vert x-x^* \Vert+\Big(\Big\vert\frac{\alpha-1}{t}-\frac{1}{\beta}\Big\vert+\Big\vert\frac{1}{\beta}-\frac{\alpha-2}{t}\Big\vert\Big)\Vert x \Vert +\Big(\frac{1}{\beta}+\Big\vert\frac{1}{t}+\frac{1}{\beta}\Big\vert\Big)\Vert y \Vert \\
    &\le \beta\frac{\delta(t)}{\gamma(t)}\Vert x^*\Vert +a(t) \Vert x \Vert + b(t)\Vert y \Vert \\
    &\le \beta\frac{\delta(t)}{\gamma(t)}\Vert x^*\Vert +K(t) \Vert (x,y) \Vert.
\end{align*}
Set $P(t)=\max\big\{\beta\frac{\delta(t)}{\gamma(t)}\Vert x^*\Vert,K(t)\big\}$. We obtain 
\begin{align*}
   \Vert F(t,x,y) \Vert \le P(t)(1+\Vert (x,y)\Vert). 
\end{align*}
As such, $P \in L^1_{loc}([t_0,+\infty))$. The proof is thereby completed.  
\end{proof}

\section{Convergence properties of $(\text{NS-HR})_{\alpha,\beta}$}
Before going into the convergence analysis of $(\text{NS-HR})_{\alpha,\beta}$, we need to make the following assumption on the values of the coefficients.
\begin{Assumption}\label{ass:B}
Let $\alpha > 3$, and $\beta>0$.  
Let $\delta,\gamma:[t_0,+\infty)\to(0,+\infty)$ be $C^1$ functions satisfying:

\begin{enumerate}[label=(\roman*)]
    \item 
    $\displaystyle \liminf_{t\to +\infty} \frac{\dot{\gamma}(t)}{t\delta(t)} > 0.$

    \item
    $\displaystyle \frac{t\delta(t)}{\gamma(t)} = O(1/t).$

    \item
    $\displaystyle \frac{\dot{\gamma}(t)}{\gamma(t)} = O(1/t).$

    \item
    There exists $0<\zeta\le\alpha-3$ such that for all large $t$,
    \[
    0 \le  \frac{t\dot{\delta}(t)}{\delta(t)} \le \alpha - 3 - \zeta.
    \]
\end{enumerate}
\end{Assumption}
\begin{remark}
\label{remark 3.1}
Assumption~\ref{ass:B} is mild. For instance, it is verified with the polynomials
\[
\delta(t)=t^{p},\qquad \gamma(t)=ct^{p+2}\qquad (t\ge t_0>0),
\]
where $c>0$ and \(p\in[0,\alpha-3)\). Indeed:
\begin{itemize}
\item[(i)] Since \(\dot\gamma(t)=c(p+2)t^{p+1}\),
\[
\frac{\dot\gamma(t)}{t\delta(t)}
=\frac{c(p+2)t^{p+1}}{t\cdot t^{p}}
=c(p+2),
\]
and therefore
\[
\liminf_{t\to +\infty}\frac{\dot\gamma(t)}{t\delta(t)}=c(p+2)>0.
\]

\item[(ii)]
\[
\frac{t\delta(t)}{\gamma(t)}
=\frac{t\cdot t^{p}}{ct^{p+2}}
=\frac{1}{ct}
=O\!\left(\frac{1}{t}\right).
\]

\item[(iii)]
\[
\frac{\dot\gamma(t)}{\gamma(t)}
=\frac{c(p+2)t^{p+1}}{ct^{p+2}}
=\frac{p+2}{t}
=O\!\left(\frac{1}{t}\right).
\]

\item[(iv)] Since \(\dot\delta(t)=pt^{p-1}\),
\[
\frac{t\dot\delta(t)}{\delta(t)}
=\frac{t\cdot pt^{p-1}}{t^{p}}
=p.
\]
Choose, for instance,
\[
\zeta:=\frac{\alpha-3-p}{2}\in(0,\alpha-3),
\]
so that \(0\le p\le \alpha-3-\zeta\) for all large \(t\).
\end{itemize}   
\end{remark}

Since $\text{argmin}_{\mathcal{H}}f \neq \emptyset$, we fix some $x^* \in \text{argmin}_{\mathcal{H}}f$. For the convergence analysis of $(\text{NS-HR})_{\alpha,\beta}$, we employ the Lyapunov function $\mathcal{V}:=\mathcal{V}_1+\mathcal{V}_{\sigma,\eta}$, where  
\begin{align}
&\mathcal{V}_1(t)=a(t)\Big(f_{\gamma(t)}(x(t))-f^*\Big),\label{V.1}\\
&\mathcal{V}_{\sigma,\eta}(t)=\frac{1}{2}\Big\Vert \sigma(x(t)-x^*)+t\dot{x}(t)+\beta t\delta(t)\nabla f_{\gamma(t)}(x(t)) \Big\Vert^2 +\eta\Vert x(t)-x^*\Vert^2, \label{V.2}
\end{align}
with $0<\sigma<\alpha-1$ being chosen later, and the coefficients $a$ and $\eta$ being defined as follows
\begin{align*}
    &a(t)=\Big[t-\beta(\sigma+1-\alpha)\Big]t\delta(t),\\
&\eta=\frac{\sigma(\alpha-\sigma-1)}{2}.
\end{align*}

We will need the following technical lemma for the convergence analysis. 
\begin{lemma}
\label{lemma: derivative-gradient}
Suppose that $\delta, \gamma:[t_0,+\infty) \to (0,+\infty)$, and $x:[t_0,+\infty)\to \mathcal{H}$ are of class $C^1$. Suppose that $f:\mathcal{H}\to\mathbb{R}\cup\{+\infty\}$ is proper, convex, and lower semicontinuous. Then, we have 
\begin{align*}
    \Big\Vert \frac{\mathrm{d}}{\mathrm{d}t}[\delta(t)\nabla f_{\gamma(t)}(x(t))] \Big \Vert \le \frac{\delta(t)}{\gamma(t)}\Vert \dot{x}(t) \Vert+\Big[\frac{2\delta(t)\vert \dot{\gamma}(t)\vert}{\gamma(t)^2}+\frac{\vert \dot{\delta}(t)\vert}{\gamma(t)}\Big]\Vert x(t)-x^* \Vert. 
\end{align*}
\end{lemma}
\begin{proof}
The lemma is a direct corollary of the more general Lemma \ref{lem:derivative of moreau}, which we will prove later. 
\end{proof}

We are now in the position to present the convergence properties of $(\text{NS-HR})_{\alpha,\beta}$.
\begin{theorem}
\label{first convergence result}
Under the standing Assumption~\ref{ass:B}, and for any solution trajectory $x:[t_0,+\infty) \to \mathcal{H}$ to $(\text{NS-HR})_{\alpha,\beta}$, we have 
\begin{enumerate}[label=(\roman*)]
               \item $\displaystyle f_{\gamma(t)}(x(t))-f^*=o\Big(\frac{1}{t^2\delta(t)}\Big).$
        \item $\displaystyle \Vert \nabla f_{\gamma(t)}(x(t))\Vert=o\Big(\frac{1}{t^2\delta(t)}\Big).$
        \item $\Vert \dot{x}(t) \Vert = o(\frac{1}{t})$.
        \item $\displaystyle \int_{t_0}^{+\infty}t\Vert \dot{x}(t)\Vert^2\mathrm{d}t < +\infty$. 
        \item $\displaystyle \int_{t_0}^{+\infty}t^3\delta(t)^2\Vert \nabla f_{\gamma(t)}(x(t))\Vert^2 \mathrm{d}t <+\infty$.
        \item if in addition $\liminf_{t\to +\infty}\frac{t^2\delta(t)}{\gamma(t)}>0$, then $x(t)$ converges weakly as $t \to +\infty$, and its limit belongs to $\text{argmin}_{\mathcal{H}}f$.
    \end{enumerate}
\end{theorem}
\begin{proof}
The proof proceeds by defining a suitable Lyapunov function and computing its derivative. By using the standing assumptions, we will then be able to bound key terms and prove items \textit{(iv)} and \textit{(v)}. We then use these bounds to prove the rates for the velocity \textit{(iii)}, the gradient \textit{(ii)}, and the functional gap \textit{(i)}. We then conclude by leveraging Opial's lemma to prove weak convergence of the trajectory \textit{(vi)}. 

As announced, we consider the following Lyapunov function $\mathcal{V}(t):= \mathcal{V}_1(t) + \mathcal{V}_{\sigma, \eta}(t)$ as defined in Equations~\eqref{V.1}-\eqref{V.2}. Along the trajectory $x(t)$, the Lyapunov function is nonnegative for sufficiently large $t$. We compute the derivative of $\mathcal{V}_1$ \begin{align*}
\dot{\mathcal{V}}_1(t)=\dot{a}(t)\Big(f_{\gamma(t)}(x(t))-f^*\Big)+a(t)\frac{\mathrm{d}}{\mathrm{d}t}f_{\gamma(t)}(x(t)).
\end{align*}
By the chain rule 
\begin{align*}
\frac{\mathrm{d}}{\mathrm{d}t}f_{\gamma(t)}(x(t)) =\dot{\gamma}(t)\left.\frac{\mathrm{d}}{\mathrm{d}\gamma}f_{\gamma}(x(t))\right|_{\gamma=\gamma(t)}+\Big\langle \dot{x}(t),\nabla f_{\gamma(t)}(x(t))
\Big\rangle.
\end{align*}
We now recall a result concerning the derivative of the Moreau envelope $f_{\gamma} $  with respect to $\gamma$ \cite{AttouchCabot2018JDE} 
\begin{align*}
    \frac{\mathrm{d}}{\mathrm{d}\gamma}f_{\gamma}(x)=-\frac{1}{2}\Vert \nabla f_{\gamma}(x) \Vert^2 \quad \forall x \in \mathcal{H.}
\end{align*}
As a result, 
\begin{align}
\label{derivative-lyapunov-first-compo}
\dot{\mathcal{V}}_1(t)=\dot{a}(t)\Big(f_{\gamma(t)}(x(t))-f^*\Big)-\frac{\dot{\gamma}(t)a(t)}{2}\Vert \nabla f_{\gamma(t)}(x(t)) \Vert^2+a(t)\Big\langle \dot{x}(t),\nabla f_{\gamma(t)}(x(t)).
\Big\rangle
\end{align}
We now turn to computing the derivative $\mathcal{V}_{\sigma,\eta}$. We have
\begin{align}  
\label{derivative-lyapunov-opt}
\dot{\mathcal{V}}_{\sigma,\eta}(t)=&\Big \langle \sigma(x(t)-x^*)+t\dot{x}(t)+\beta t\delta(t)\nabla f_{\gamma(t)}(x(t)), (\sigma+1)\dot{x}(t)+t\ddot{x}(t)+\beta\frac{\mathrm{d}}{\mathrm{d}t}\Big[t\delta(t)\nabla f_{\gamma(t)}(x(t))\Big] \Big\rangle \notag\\
    &+2\eta\langle x(t)-x^*,\dot{x}(t) \rangle. 
\end{align}
From the expression of $(\text{NS-HR})_{\alpha,\beta}$, we can express the term $t \ddot{x}(t)$ as
\begin{align*}
    t\ddot{x}(t)=-\alpha\dot{x}(t)-\beta t\frac{\mathrm{d}}{\mathrm{d}t}\Big[\delta(t)\nabla f_{\gamma(t)}(x(t))\Big]-(t+\beta)\delta(t)\nabla f_{\gamma(t)}(x(t)),
\end{align*}
and by substituting this equivalence into \eqref{derivative-lyapunov-opt} and simplifying, we obtain
\begin{align*}
\dot{\mathcal{V}}_{\sigma,\eta}(t)=&\Big \langle \sigma(x(t)-x^*)+t\dot{x}(t)+\beta t\delta(t)\nabla f_{\gamma(t)}(x(t)), (\sigma+1-\alpha)\dot{x}(t)-t\delta(t)\nabla f_{\gamma(t)}(x(t)) \Big\rangle \\
    &+2\eta\langle x(t)-x^*,\dot{x}(t) \rangle. 
\end{align*}
Further algebraic manipulations yield 
\begin{align*}
\dot{\mathcal{V}}_{\sigma,\eta}(t)=&\Big[\sigma(\sigma+1-\alpha)+2\eta\Big]\langle x(t)-x^*,\dot{x}(t) \rangle +t(\sigma+1-\alpha)\Vert \dot{x}(t) \Vert^2 \\
    &+\Big[\beta(\sigma+1-\alpha)-t\Big]t\delta(t)\Big\langle \nabla f_{\gamma(t)}(x(t)),\dot{x}(t) \Big\rangle -\sigma t\delta(t) \Big \langle\nabla f_{\gamma(t)}(x(t)),x(t)-x^{*} \Big \rangle \\
    &-\beta t^2\delta(t)^2\Big\Vert \nabla f_{\gamma(t)}(x(t)) \Big \Vert^2.
\end{align*}
From the convexity of $f_{\gamma(t)}$, we know that
\begin{align*}
\Big \langle\nabla f_{\gamma(t)}(x(t)),x(t)-x^{*} \Big \rangle \geq f_{\gamma(t)}(x(t))-f^*. 
\end{align*}
As such, we can upper bound $\dot{\mathcal{V}}_{\sigma,\eta}(t)$ as 
\begin{align}
    \label{derivative-lyapunov-second-compo}
    \begin{split}
     \dot{\mathcal{V}}_{\sigma,\eta}(t) \le& \Big[\sigma(\sigma+1-\alpha)+2\eta\Big]\langle x(t)-x^*,\dot{x}(t) \rangle +t(\sigma+1-\alpha)\Vert \dot{x}(t) \Vert^2 \\
    &+\Big[\beta(\sigma+1-\alpha)-t\Big]t\delta(t)\Big\langle \nabla f_{\gamma(t)}(x(t)),\dot{x}(t) \Big\rangle  -\sigma t\delta(t) \Big(f_{\gamma(t)}(x(t))-f^*\Big) \\
    &-\beta t^2\delta(t)^2\Big\Vert \nabla f_{\gamma(t)}(x(t)) \Big \Vert^2.   
    \end{split}
\end{align}
By adding the expressions~\eqref{derivative-lyapunov-first-compo} with \eqref{derivative-lyapunov-second-compo}, we arrive at the bound,
\begin{align*}
\dot{\mathcal{V}}(t) \le & \Big(\dot{a}(t)-\sigma t\delta(t)\Big)\Big(f_{\gamma(t)}(x(t))-f^*\Big)+\Big[\sigma(\sigma+1-\alpha)+2\eta\Big]\langle x(t)-x^*,\dot{x}(t) \rangle \\
&+t(\sigma+1-\alpha)\Vert \dot{x}(t) \Vert^2-\Big(\beta t^2\delta(t)^2+\frac{\dot{\gamma}(t)a(t)}{2}\Big)\Vert \nabla f_{\gamma(t)}(x(t)) \Vert^2\\
&+\Big\{a(t)+\Big[\beta(\sigma+1-\alpha)-t\Big]t\delta(t) \Big\}\Big\langle \nabla f_{\gamma(t)}(x(t)),\dot{x}(t) \Big\rangle.
\end{align*}
Substituting now the definition of $a(t)$ and $\eta$, 
\begin{align*}
    \dot{\mathcal{V}}(t) \le  \Big(\dot{a}(t)-\sigma t\delta(t)\Big)\Big(f_{\gamma(t)}(x(t))-f^*\Big) 
+t(\sigma+1-\alpha)\Vert \dot{x}(t) \Vert^2\\
-\Big(\beta t^2\delta(t)^2+\frac{\dot{\gamma}(t)a(t)}{2}\Big)\Vert \nabla f_{\gamma(t)}(x(t)) \Vert^2.
\end{align*}
From Assumption \ref{ass:B}\textit{(iv)}, one can easily show that for a suitably chosen $\sigma \in (0,\alpha-1)$, we have $\dot{a}(t)-\sigma t\delta(t) \le 0$ for sufficiently large $t$. Indeed, 
\[
a(t)=\big[t-\beta(\sigma+1-\alpha)\big]\,t\,\delta(t)
=\big(t^2-\beta(\sigma+1-\alpha)t\big)\delta(t).
\]
Differentiating, we obtain
\[
\dot a(t)
=\big(2t-\beta(\sigma+1-\alpha)\big)\delta(t)
+\big(t^2-\beta(\sigma+1-\alpha)t\big)\dot\delta(t).
\]
Hence
\[
\dot a(t)-\sigma t\delta(t)
=
\Big((2-\sigma)t-\beta(\sigma+1-\alpha)\Big)\delta(t)
+\big(t^2-\beta(\sigma+1-\alpha)t\big)\dot\delta(t).
\]
Dividing by \(t\delta(t)\), we get
\[
\frac{\dot a(t)-\sigma t\delta(t)}{t\delta(t)}
=
(2-\sigma)-\frac{\beta(\sigma+1-\alpha)}{t}
+\Big(t-\beta(\sigma+1-\alpha)\Big)\frac{\dot\delta(t)}{\delta(t)}.
\]
Equivalently,
\[
\frac{\dot a(t)-\sigma t\delta(t)}{t\delta(t)}
=
(2-\sigma)-\frac{\beta(\sigma+1-\alpha)}{t}
+\left(1-\frac{\beta(\sigma+1-\alpha)}{t}\right)\frac{t\dot\delta(t)}{\delta(t)}.
\]

By Assumption \ref{ass:B}\textit{(iv)}, there exists \(0<\zeta\le\alpha-3\) such that for all sufficiently large \(t\),
\[
0\le \frac{t\dot\delta(t)}{\delta(t)}\le \alpha-3-\zeta.
\]
Therefore, for all sufficiently large \(t\),
\[
\frac{\dot a(t)-\sigma t\delta(t)}{t\delta(t)}
\le
(2-\sigma)-\frac{\beta(\sigma+1-\alpha)}{t}
+\left(1-\frac{\beta(\sigma+1-\alpha)}{t}\right)(\alpha-3-\zeta).
\]
Letting \(t\to+\infty\), the right-hand side tends to
\[
(2-\sigma)+(\alpha-3-\zeta)=\alpha-1-\zeta-\sigma.
\]
Hence, if we choose
\[
\sigma>\alpha-1-\zeta,
\]
then
\[
\alpha-1-\zeta-\sigma<0.
\]
Since \(0<\zeta\le\alpha-3\), the interval
\[
(\alpha-1-\zeta,\alpha-1)
\]
is nonempty. Thus we may choose
\[
\sigma\in(\alpha-1-\zeta,\alpha-1)\subset(0,\alpha-1).
\]
For such a choice of \(\sigma\), we obtain
\[
\frac{\dot a(t)-\sigma t\delta(t)}{t\delta(t)}\le 0
\]
for all sufficiently large \(t\). Since \(t\delta(t)>0\) for all \(t\), it follows that
\[
\dot a(t)-\sigma t\delta(t)\le 0
\]
for all sufficiently large \(t\).\\ 
As a result, for sufficiently large $t$,
\begin{align}
\label{derivative-lyapunov-less-than-xdot-gradient}
    \dot{\mathcal{V}}(t) \le  
t(\sigma+1-\alpha)\Vert \dot{x}(t) \Vert^2
-\Big(\beta t^2\delta(t)^2+\frac{\dot{\gamma}(t)a(t)}{2}\Big)\Vert \nabla f_{\gamma(t)}(x(t)) \Vert^2.
\end{align}
We have $\liminf_{t \to +\infty}\frac{\dot{\gamma}(t)}{t\delta(t)}>0$ by Assumption \ref{ass:B}\textit{(i)}, $\delta(t)>0 \quad \forall t \geq t_0$, and $a(t)>0$ for sufficiently large $t$. With these in mind and the fact that $\mathcal{V}(t) \geq 0$ for sufficiently large $t$, we can integrate \eqref{derivative-lyapunov-less-than-xdot-gradient} and obtain 
\begin{align}
\label{first integral estimate}
    &\int_{t_0}^{+\infty} t\Vert \dot{x}(t)\Vert^2 \mathrm{d}t <+\infty, \\
    \label{almost second integral estimate}
    &\int_{T}^{+\infty} \Big(\beta t^2\delta(t)^2+\frac{\dot{\gamma}(t)a(t)}{2}\Big)\Vert \nabla f_{\gamma(t)}(x(t)) \Vert^2\mathrm{d}t <+\infty \quad 
    \text{for some } T>0 \text{ large enough}. 
\end{align}
Here \eqref{first integral estimate} is exactly \emph{item \textit{(iv)}}. \\
We have
\begin{align*}
\beta t^2\delta(t)^2+\frac{\dot{\gamma}(t)a(t)}{2}=\beta t^2\delta(t)^2+\frac{\dot{\gamma}(t)[t-\beta(\sigma+1-\alpha)]t\delta(t)}{2}.  
\end{align*}
As a result, 
\begin{align*}
    \liminf_{t \to +\infty} \frac{\beta t^2\delta(t)^2+\frac{\dot{\gamma}(t)a(t)}{2}}{t^3 \delta(t)^2}&=\liminf_{t \to +\infty} \frac{\beta}{t}+\frac{\dot{\gamma}(t)[t-\beta(\sigma+1-\alpha)]}{2t^2\delta(t)}\\
    &=\liminf_{t \to +\infty}\frac{\dot{\gamma}(t)}{2t\delta(t)} > 0 \quad (\text{by Assumption \ref{ass:B}\textit{(i)}}).
\end{align*}
Combining this with \eqref{almost second integral estimate} gives 
\begin{align}
\label{ine: item iv}
\displaystyle \int_{t_0}^{+\infty}t^3\delta(t)^2\Vert \nabla f_{\gamma(t)}(x(t))\Vert^2 \mathrm{d}t <+\infty, 
\end{align}
which \emph{proves \textit{(v)}}. \\
From the upper bound on the Lyapunov derivative~\eqref{derivative-lyapunov-less-than-xdot-gradient}, and the reasoning above, we infer that $\mathcal{V}$ is also upper bounded. Moreover, from the definition of the Lyapunov function $\mathcal{V}$, we can further imply that 
\begin{align}
    &\sup_{t \geq t_0} \Vert x(t) \Vert <+\infty, \\
    &\sup_{t\geq t_0} \Vert t\dot{x}(t)+\beta t\delta(t)\nabla f_{\gamma(t)}(x(t))\Vert <+\infty. \label{dummy00}
\end{align}

We can then use the triangle inequality as,
\begin{align*}
\Vert t\dot{x}(t)\Vert \le \Vert t\dot{x}(t)+\beta t\delta(t)\nabla f_{\gamma(t)}(x(t))\Vert + \Vert -\beta t\delta(t)\nabla f_{\gamma(t)}(x(t)) \Vert.
\end{align*}
By the $\frac{1}{\gamma(t)}$-Lipschitz continuity of $\nabla f_{\gamma(t)}$, the second term is finite, 
\begin{align*}
    \sup_{t\geq t_0}\Vert t\delta(t)\nabla f_{\gamma(t)}(x(t))\Vert \le \sup_{t \geq t_0}t\delta(t)\frac{1}{\gamma(t)}\Vert x(t)-x^* \Vert <+\infty,
\end{align*}
where the last inequality holds since $\sup_{t \geq t_0} \Vert x(t) \Vert <+\infty$ and $\frac{t\delta(t)}{\gamma(t)}=O(1/t)$ (by Assumption \ref{ass:B}\textit{(ii)}). 

As such and by~\eqref{dummy00}, the left hand side is also finite, 
\begin{align*}
    \sup_{t \geq t_0} \Vert t \dot{x}(t) \Vert <+\infty, \text{ or in other words }  \Vert \dot{x}(t)\Vert = O(1/t).
\end{align*}
From Lemma \ref{lemma: derivative-gradient}, we imply 
\begin{align*}
    t^3\Big\Vert \frac{\mathrm{d}}{\mathrm{d}t}[\delta(t)\nabla f_{\gamma(t)}(x(t))] \Big \Vert^2 &\le 2t^3\frac{\delta(t)^2}{\gamma(t)^2}\Vert \dot{x}(t) \Vert^2+2t^3\Big[\frac{2\delta(t)\vert \dot{\gamma}(t)\vert}{\gamma(t)^2}+\frac{\vert \dot{\delta}(t)\vert}{\gamma(t)}\Big]^2\Vert x(t)-x^*\Vert^2.
\end{align*}
Since we have $\frac{t\delta (t)}{\gamma(t)}=O(1/t)$ (Assumption \ref{ass:B}\textit{(ii)}) and $\Vert \dot{x}(t) \Vert=O(1/t)$, 
\begin{align*}
t^3\frac{\delta(t)^2}{\gamma(t)^2}\Vert \dot{x}(t) \Vert^2=O\Big(\frac{1}{t^3}\Big). 
\end{align*}
Additionally, from Assumption \ref{ass:B}\textit{(iv)}, we have that, there exists $0<\zeta\le\alpha-3$ such that 
    \begin{align*}
0 \le \frac{t\dot{\delta}(t)}{\delta(t)} \le \alpha-3-\zeta \quad\text{ for large } t.
\end{align*}
It follows that 
\begin{align}
\label{ine: upper bound of derivative of delta}
 \vert \dot{\delta}(t) \vert \le (\alpha-3-\zeta)\frac{\delta(t)}{t} \text{ for large } t.
\end{align}
As a result, for sufficiently large $t$ 
\begin{align}
\label{ine: bounding hessian}
    t^3\Big\Vert\frac{\mathrm{d}}{\mathrm{d}t}[\delta(t)\nabla f_{\gamma(t)}(x(t))] \Big \Vert^2 &\le O\Big(\frac{1}{t^3}\Big)+2t^3\Big[\frac{2\delta(t)\vert \dot{\gamma}(t)\vert}{\gamma(t)^2}+(\alpha-3-\zeta)\frac{\delta(t)}{t\gamma(t)}\Big]^2\Vert x(t)-x^*\Vert^2  \\
    &=O\Big(\frac{1}{t^3}\Big)+2t^3\Big[O\Big(\frac{1}{t^3}\Big)+O\Big(\frac{1}{t^3}\Big)\Big]^2\Vert x(t)-x^*\Vert^2\label{dummy01}\\
    &=O\Big(\frac{1}{t^3}\Big),
\end{align}
where equality~\eqref{dummy01} comes from the fact that $\sup_{t \geq t_0}\Vert x(t)\Vert <+\infty$, $\frac{t\delta(t)}{\gamma(t)}=O(\frac{1}{t})$ (Assumption \ref{ass:B}\textit{(ii)}), and $\frac{\vert \dot{\gamma}(t)\vert}{\gamma(t)}=O(\frac{1}{t})$ (Assumption \ref{ass:B}\textit{(iii)} and the fact that $\dot{\gamma}(t)>0$ for large $t$). 

By the definition of $(\text{NS-HR})_{\alpha,\beta}$, 
\begin{align*}
    \ddot{x}(t)=-\frac{\alpha}{t}\dot x(t)-\beta\,\frac{\mathrm{d}}{\mathrm{d}t}\!\Big[\delta(t)\,\nabla f_{\gamma(t)}(x(t))\Big]
-\Big(1+\frac{\beta}{t}\Big)\delta(t)\,\nabla f_{\gamma(t)}(x(t)).
\end{align*}
Hence, by taking the norm of the left-hand side, squaring it, substituting in the right-hand side, and by using the fact that $\|a+b+c\|^2 \leq 3 (\|a\|^2 + \|b\|^2 + \|c\|^2)$, we obtain
\begin{align}
\label{ine: summability of acceleration}
    t^3\Vert\ddot{x}(t)\Vert^2 \le 3\alpha^2t \Vert \dot{x}(t) \Vert^2+3\beta^2t^3\Big\Vert \frac{\mathrm{d}}{\mathrm{d}t}[\delta(t)\nabla f_{\gamma(t)}(x(t))] \Big \Vert^2+3t(t+\beta)^2\delta(t)^2\Vert \nabla f_{\gamma(t)}(x(t)) \Vert^2.
\end{align}
From \eqref{first integral estimate}, \eqref{ine: bounding hessian}, and \eqref{ine: item iv}, we deduce that the right-hand side of the above inequality belongs to $L^1([t_0,+\infty))$. Therefore,
\begin{align}\label{dummy02}
 t^3\Vert\ddot{x}(t)\Vert^2 \in L^1([t_0,+\infty)).   
\end{align}
We are now ready to \emph{prove \textit{(iii)}} which is that the rate of convergence for $\dot{x}(t)$ is actually $\Vert\dot{x}(t)\Vert=o\big(\frac{1}{t}\big)$. \\
From \eqref{first integral estimate}, we have $\int_{t_0}^{+\infty} t\Vert \dot{x}(t) \Vert^2 \mathrm{d}t < +\infty$. So $\liminf_{t\rightarrow+\infty}t\Vert \dot{x}(t) \Vert =0$. Hence, we only need to prove that $\lim_{t \rightarrow+\infty}t\Vert \dot{x}(t) \Vert$ exists. Indeed, 
\begin{align*}
\frac{\mathrm{d}}{\mathrm{d}t} t^2\Vert\dot{x}(t)\Vert^2 &=2t\Vert \dot{x}(t) \Vert^2+2t^2\langle \dot{x}(t),\ddot{x}(t)\rangle \\
&\le 2t\Vert \dot{x}(t) \Vert^2+2\langle \sqrt{t}\dot{x}(t),(t\sqrt{t})\ddot{x}(t)\rangle \quad \textrm{(use now: $2ab \leq a^2 + b^2$)}\\
&\le3t\Vert \dot{x}(t) \Vert^2 + t^3\Vert \ddot{x}(t) \Vert^2.
\end{align*}

From \eqref{first integral estimate} and \eqref{dummy02}, we deduce that the right-hand side of the above inequality belongs to $L^1([t_0,+\infty))$. This implies that $\lim_{t \rightarrow+\infty}t^2\Vert \dot{x}(t) \Vert^2$ exists, and so does $\lim_{t \rightarrow+\infty}t\Vert \dot{x}(t) \Vert$, and we prove the claim. 

We now turn to \emph{proving \textit{(ii)}} which is to prove that $\Big\Vert \nabla f_{\gamma(t)}(x(t)) \Big\Vert=o\Big(\frac{1}{t^2\delta(t)}\Big)$. \\
Recall that from \eqref{ine: item iv}, we have 
\begin{align*}
    \liminf_{t \rightarrow+\infty} t^2\delta(t)\Big\Vert \nabla f_{\gamma(t)}(x(t)) \Big\Vert=0.
\end{align*}
Therefore, we only need to show that $\lim_{t \rightarrow+\infty} t^2\delta(t)\Big\Vert \nabla f_{\gamma(t)}(x(t)) \Big\Vert$ exists. To this end, let us set
\begin{align*}
    \xi(t)=\Vert t^2\delta(t)\nabla f_{\gamma(t)}(x(t)) \Vert^2.
\end{align*}
We have 
\begin{align*}
    \frac{\mathrm{d}}{\mathrm{d}t}\xi(t)
    &=2\Big\langle t^2\delta(t)\nabla f_{\gamma(t)}(x(t)), 2t\delta(t)\nabla f_{\gamma(t)}(x(t))+t^2\frac{\mathrm{d}}{\mathrm{d}t}\delta(t)\nabla f_{\gamma(t)}(x(t)) \Big\rangle \\
    &=4t^3\delta(t)^2\Vert \nabla f_{\gamma(t)}(x(t)) \Vert^2+2t^4\delta(t)\Big\langle \nabla f_{\gamma(t)} (x(t)),\frac{\mathrm{d}}{\mathrm{d}t}\delta(t)\nabla f_{\gamma(t)}(x(t)) \Big\rangle\\
    &\le 4t^3\delta(t)^2\Vert \nabla f_{\gamma(t)}(x(t)) \Vert^2+2t^4\delta(t)\Vert \nabla f_{\gamma(t)} (x(t))\Vert \Big\Vert\frac{\mathrm{d}}{\mathrm{d}t}\delta(t)\nabla f_{\gamma(t)}(x(t)) \Big\Vert \\
    &\le  4t^3\delta(t)^2\Vert \nabla f_{\gamma(t)}(x(t)) \Vert^2+\frac{2t^4\delta(t)^2}{\gamma(t)}\Vert \nabla f_{\gamma(t)}(x(t)) \Vert \Vert \dot{x}(t) \Vert+\\&\hskip1cm +\Big(\frac{4t^4\delta(t)^2\vert \dot{\gamma}(t)\vert}{\gamma(t)}+2t^4\delta(t)\vert \dot{\delta}(t)\vert\Big)\Vert \nabla f_{\gamma(t)}(x(t)) \Vert^2,
\end{align*}
where the last inequality is true because of Lemma \ref{lemma: derivative-gradient}. 

From \eqref{ine: item iv}, we have 
\begin{align*}
4t^3\delta(t)^2\Vert \nabla f_{\gamma(t)}(x(t)) \Vert^2 \in L^1([t_0,+\infty)).   
\end{align*}
Besides, from Assumption \ref{ass:B}\textit{(ii)}, \textit{(iii)}, and \eqref{ine: upper bound of derivative of delta}, we have 
\begin{align*}
    \Big(\frac{4t^4\delta(t)^2\vert \dot{\gamma}(t)\vert}{\gamma(t)}+2t^4\delta(t)\vert \dot{\delta}(t)\vert\Big)\Vert \nabla f_{\gamma(t)}(x(t)) \Vert^2 = O\Big(t^3\delta(t)^2\Vert \nabla f_{\gamma(t)}(x(t)) \Vert^2\Big),
\end{align*}
which belongs to $L^1([t_0,+\infty))$ due to \eqref{ine: item iv}.

Additionally,  
\begin{align*}
\frac{2t^4\delta(t)^2}{\gamma(t)}\Vert \nabla f_{\gamma(t)}(x(t)) \Vert \Vert \dot{x}(t) \Vert &=\frac{2t^4\delta(t)^2}{\gamma(t)}\sqrt{\frac{\gamma(t)}{t}}\Vert \nabla f_{\gamma(t)}(x(t)) \Vert \sqrt{\frac{t}{\gamma(t)}}\Vert \dot{x}(t) \Vert \\
&\le \frac{t^4\delta(t)^2}{\gamma(t)}\Big(\frac{\gamma(t)}{t}\Vert \nabla f _{\gamma(t)}(x(t)) \Vert^2+\frac{t}{\gamma(t)}\Vert \dot{x}(t) \Vert^2\Big) \\
&=t^3\delta(t)^2\Vert \nabla f _{\gamma(t)}(x(t)) \Vert^2+\frac{t^5\delta(t)^2}{\gamma(t)^2}\Vert \dot{x}(t) \Vert^2 \\
&=t^3\delta(t)^2\Vert \nabla f _{\gamma(t)}(x(t)) \Vert^2+O(t\Vert \dot{x}(t) \Vert^2) \text{ (thanks to Assumption \ref{ass:B}\textit{(ii)})}
\end{align*}
This, combined with \eqref{ine: item iv} and \eqref{first integral estimate}, implies 
\begin{align*}
 \frac{2t^4\delta(t)^2}{\gamma(t)}\Vert \nabla f_{\gamma(t)}(x(t)) \Vert \Vert \dot{x}(t) \Vert  \in L^1([t_0,+\infty)).   
\end{align*}
As a result, 
\begin{align*}
    \frac{\mathrm{d}}{\mathrm{d}t}\xi(t) 
    &\le 4t^3\delta(t)^2\Vert \nabla f_{\gamma(t)}(x(t)) \Vert^2+\frac{2t^4\delta(t)^2}{\gamma(t)}\Vert \nabla f_{\gamma(t)}(x(t)) \Vert \Vert \dot{x}(t) \Vert+\\&\hskip1cm+\Big(\frac{4t^4\delta(t)^2\vert \dot{\gamma}(t)\vert}{\gamma(t)}+2t^4\delta(t)\vert \dot{\delta}(t)\vert\Big)\Vert \nabla f_{\gamma(t)}(x(t)) \Vert^2 
    \in L^1([t_0,+\infty)). 
\end{align*}
This classically implies that $\lim_{t \to +\infty}\xi(t)$ exists. We therefore complete the proof for \textit{(ii)}. 

We have now that by the convexity of $f_{\gamma(t)}$
\begin{align*}
    f_{\gamma(t)}(x(t))-f(x^*)&=f_{\gamma(t)}(x(t))-f_{\gamma(t)}(x^*) \le \langle \nabla f_{\gamma(t)}(x(t)),x(t)-x^* \rangle \\
    &\le \Vert \nabla f_{\gamma(t)}(x(t)) \Vert \Vert x(t)-x^* \Vert.
\end{align*}
Since $t \mapsto x(t)$ is bounded and $\Vert \nabla f_{\gamma(t)}(x(t)) \Vert=o\Big(\frac{1}{t^2\delta(t)}\Big)$, we have 
\begin{align*}
f_{\gamma(t)}(x(t))-f(x^*)=  o\Big(\frac{1}{t^2\delta(t)}\Big),
\end{align*}
which \emph{proves \textit{(i)}}. 

What remains is to \emph{prove \textit{(vi)}} which is {the weak convergence of $x(t)$ to a minimizer of $f$ as $t \to +\infty$}.\\[2mm]
We will rely on the Opial's lemma \cite{Opial}, which we report here with our notation for completeness.

\begin{lemma}[Opial's lemma] Let $\mathcal{H}$ be a Hilbert space, and $S \subset \mathcal{H}$. Consider a mapping $x:[0,+\infty) \longrightarrow \mathcal{H}$. Asume the following 
\begin{enumerate}[label=(\roman*)]
    \item for any $x \in S$, $\lim_{t \rightarrow +\infty} \Vert x(t)-x \Vert$ exists. 
    \item Each weak sequential cluster point of the map $x$ is an element of $S$.
\end{enumerate}
Then $x(t)$ converges weakly to some element $x_{\infty} \in S$. 
\end{lemma}

To check that we verify the first item of the Opial's lemma, let us consider $h:[t_0,+\infty) \to \mathbb
{R}^n$ defined by 
\begin{align*}
    h(t)=\frac{1}{2}\Vert x(t)-x^* \Vert^2,
\end{align*}
where $x^*$ is an arbitrary element of $\text{argmin}_{\mathcal{H}}f$. Simple computation yields 
\begin{align*}
    &\dot{h}(t)=\langle x(t)-x^*, \dot{x}(t)\rangle. \\
    &\ddot{h}(t)=\Vert \dot{x}(t) \Vert^2+\langle x(t)-x^*,\ddot{x}(t)\rangle.
\end{align*}
As a result,
\begin{align*}
    \ddot{h}(t)+\frac{\alpha}{t}\dot{h}(t)=\Vert \dot{x}(t) \Vert^2 +\langle x(t)-x^*,\ddot{x}(t)+\frac{\alpha}{t}\dot{x}(t) \rangle.
\end{align*}
Using the definition of $(\text{NS-HR})_{\alpha,\beta}$ and the above equality, we deduce 
\begin{align*}
    \ddot{h}(t)+\frac{\alpha}{t}\dot{h}(t)+\Big(1+\frac{\beta}{t}\Big)\delta(t)\langle\nabla f_{\gamma(t)}(x(t)), x(t)-x^*\rangle =\Vert\dot{x}(t) \Vert^2-\beta \Big\langle x(t)-x^*,\frac{\mathrm{d}}{\mathrm{d}t}\Big[\delta(t)\nabla f_{\gamma(t)}(x(t))\Big] \Big \rangle.
\end{align*}
Recall that $\nabla f_{\gamma(t)}$ is $\gamma(t)-$cocoercive since $f_{\gamma(t)}$ is convex has $\frac{1}{\gamma(t)}-$Lipschitz gradient, that is $\langle \nabla f_{\gamma(t)}(x(t)) - \nabla f_{\gamma(t)}(x^*),x(t)-x^*\rangle \geq \gamma(t) \Vert \nabla f_{\gamma(t)}(x(t)) - \nabla f_{\gamma(t)}(x^*) \Vert^2$. Using this fact and the Cauchy-Schwarz inequality yields 
\begin{align*}
    t\ddot{h}(t)+\alpha \dot{h}(t)+(t+\beta)\gamma(t)\delta(t)\Vert \nabla f_{\gamma(t)}(x(t)) \Vert^2 \le t\Vert \dot{x}(t) \Vert^2 +\beta t\Vert x(t)-x^* \Vert \Big \Vert \frac{\mathrm{d}}{\mathrm{d}t}\Big[\delta(t)\nabla f_{\gamma(t)}(x(t))\Big] \Big \Vert.
\end{align*}

Set 
\begin{align*}
    &\theta(t)=(t+\beta)\gamma(t)\delta(t)\Vert \nabla f_{\gamma(t)}(x(t)) \Vert^2.\\
    &k(t)=t\Vert \dot{x}(t) \Vert^2 +\beta t\Vert x(t)-x^* \Vert \Big \Vert \frac{\mathrm{d}}{\mathrm{d}t}\Big[\delta(t)\nabla f_{\gamma(t)}(x(t))\Big] \Big \Vert.
\end{align*}
Recall that $x$ is bounded, $t\Vert \dot{x}(t) \Vert^2 \in L^1([t_0,+\infty))$ and 
\begin{align*}
\Big \Vert \frac{\mathrm{d}}{\mathrm{d}t}\Big[\delta(t)\nabla f_{\gamma(t)}(x(t))\Big] \Big \Vert ={O}\Big(\frac{1}{t^3}\Big). 
\end{align*}
We obtain $k \in L^1([t_0,+\infty),\mathbb{R}^+)$. Moreover, $\theta$ is a nonnegative function. We are now in the position to conclude that $\lim_{t \to +\infty}h(t)$ exists by invoking the following lemma \cite{AttouchPeyp2019b}.
\begin{lemma}
    Suppose $h:[t_0,+\infty) \to \mathbb{R}$ is a $C^1$ function and bounded from below, where $t_0>0$. Suppose $\alpha>1, \theta:[t_0,+\infty) \to \mathbb{R}^+$ and $k \in L^1([t_0,+\infty),\mathbb{R}^+)$ satisfy
    \begin{align*}
        t\ddot{h}(t)+\alpha\dot{h}(t)+\theta(t)\le k(t)   \text{ for a.e. } x \in [t_0, +\infty). 
    \end{align*}
Then, $\lim_{t \to +\infty}h(t)$ exists.
\end{lemma}
\medskip
We now turn to check the second item of the Opial's lemma. Let $t_n \to +\infty$ and assume that $t_n \in [t_0, +\infty)$ and $x(t_n) \rightharpoonup  x^*$ as $n \to +\infty$. This means that $\langle x(t_n),y\rangle \to \langle x^*,y \rangle$ as $n \to +\infty$ for every $y \in \mathcal{H}$. We need to show that $0 \in \partial f(x^*)$.

It follows from $\Big\Vert \nabla f_{\gamma(t)}(x(t)) \Big\Vert=o\Big(\frac{1}{t^2\delta(t)}\Big)$ that
\begin{align*}
    \lim_{t \to +\infty}\frac{t^2\delta(t)}{\gamma(t)}\Big(x(t)-\text{prox}_{\gamma(t)f}(x(t)\Big)=0.
\end{align*}
Set $y(t):=x(t)-\text{prox}_{\gamma(t)f}(x(t))$. Since $$\liminf_{t\to +\infty}\frac{t^2\delta(t)}{\gamma(t)}>0,$$ it follows that $\lim_{t \to +\infty} y(t)=0$ and by the definition of proximal operators
\begin{align*}
    \frac{y(t)}{\gamma(t)} \in  \partial f\Big(x(t)-y(t)\Big).
\end{align*}
In particular, 
\begin{align*}
    \frac{y(t_n)}{\gamma(t_n)} \in  \partial f\Big(x(t_n)-y(t_n)\Big).
\end{align*}
We have $\frac{y(t_n)}{\gamma(t_n)}$ converges strongly to zero and $x(t_n)-y(t_n)$ converges weakly to $x^*$. Therefore, by the weak-strong closedness of the graph of $\partial f$, we imply that $0 \in \partial f(x^*)$. The second condition of the Opial's lemma is thereby verified. 

All in all, we have proved \textit{(vi)}, via the Opial's lemma, that is that $x(t)$ converges weakly as $t \to +\infty$, and its weak limit belongs to $\text{argmin}_{\mathcal{H}}f$.
\end{proof}
\begin{remark}
\label{cor:prox-rate}
Under the assumptions of Theorem~\ref{first convergence result}, one can also obtain the convergence rate for the original objective function:
\[
f\Big(\text{prox}_{\gamma(t)f}(x(t))\Big)-f^*
=
o\Big(\frac{1}{t^2\delta(t)}\Big).
\]
Indeed, set
\[
p(t):=\text{prox}_{\gamma(t)f}(x(t)).
\]
By the definition of the Moreau envelope,
\[
f_{\gamma(t)}(x(t))-f^*
=
f(p(t))-f^*+\frac{1}{2\gamma(t)}\|x(t)-p(t)\|^2 \geq f(p(t))-f^*.
\]
The conclusion now follows from Theorem~\ref{first convergence result}(i).
\end{remark}
\begin{remark}
\label{rem:arbitrarily-fast}
In the polynomial regime of Remark~\ref{remark 3.1}, namely
\[
\delta(t)=t^p,\qquad \gamma(t)=ct^{p+2},
\qquad c>0,\qquad p\in[0,\alpha-3),
\]
Theorem~\ref{first convergence result} together with Remark~\ref{cor:prox-rate} yields
\[
f_{\gamma(t)}(x(t))-f^*
=o(t^{-(p+2)}),
\]
\[
f\bigl(\text{prox}_{\gamma(t)f}(x(t))\bigr)-f^*
=o(t^{-(p+2)}),
\]
and
\[
\|\nabla f_{\gamma(t)}(x(t))\|
=o(t^{-(p+2)}).
\]
Since \(p\) can be chosen arbitrarily close to \(\alpha-3\), the exponent \(p+2\) can be made arbitrarily close to \(\alpha-1\). Hence, by increasing \(\alpha\), the dynamic can achieve arbitrarily fast polynomial convergence rates.
\end{remark}
\begin{remark}
\label{rem:comparison-bot}
It is instructive to compare the admissible polynomial regime of the present paper with that considered by Bot and Karapetyants~\cite{BotKarapetyants2022}. In our setting, Remark~\ref{remark 3.1} allows
\[
\delta(t)=t^p,\qquad \gamma(t)=ct^{p+2},
\qquad c>0,\qquad p\in[0,\alpha-3).
\]
Hence the Moreau parameter \(\gamma(t)\) grows like a power \(t^{p+2}\) with exponent larger or equal to \(2\), and this exponent can be made arbitrarily large by increasing \(\alpha\).

On the other hand, in the polynomial setting analyzed in~\cite{BotKarapetyants2022}, the Moreau parameter is taken of the form
\[
\lambda(t)=\lambda t^l,
\]
with \(0\le l\le 1\). Therefore, the admissible smoothing regime in~\cite{BotKarapetyants2022} is fundamentally different from the one considered here: their theory allows at most polynomial growth of order \(t\), whereas our framework naturally covers Moreau parameters growing like \(t^r\) with \(r\geq2\), and even arbitrarily large \(r\) as \(\alpha\) increases. In this sense, the two models operate in distinct asymptotic regimes.
\end{remark}

\section{Extension to the setting of maximally monotone operators}
A broad and unifying way to phrase many optimization and variational problems is the following: given a real Hilbert space $\mathcal{H}$ and a maximally monotone operator $A : \mathcal{H} \to 2^{\mathcal{H}}$, find a point $x^\star \in \mathcal{H}$ such that
\[
0 \in A x^\star .
\]
This zero-finding formulation encompasses first-order optimality conditions for convex minimization (for instance, when $A=\partial f$), monotone inclusions arising from constrained problems and KKT systems, as well as equilibrium and variational inequality models. Maximal monotonicity is a natural regularity assumption guaranteeing that the resolvent operator is everywhere defined and single-valued, which in turn allows the use of proximal and splitting methods such as the proximal point algorithm, the forward--backward scheme, and the Douglas--Rachford method.

To treat nonregularity in this case we resort to Yosida approximation. Let $A : \mathcal{H}\to 2^{\mathcal{H}}$ be a maximally monotone operator and let $\lambda>0$. The resolvent of $A$ is defined by
\[
J_{\lambda A} := (I+\lambda A)^{-1}.
\]
The Yosida approximation of $A$ is the single-valued operator $A_\lambda : \mathcal{H} \to \mathcal{H}$ given by
\[
A_\lambda := \frac{1}{\lambda}\bigl(I - J_{\lambda A}\bigr),
\qquad \text{that is,} \qquad
A_\lambda(x) = \frac{1}{\lambda}\bigl(x - J_{\lambda A}(x)\bigr).
\]
Equivalently, one has
\[
A_\lambda(x) \in A\bigl(J_{\lambda A}(x)\bigr)
\quad \text{and} \quad
J_{\lambda A}(x) = x - \lambda A_\lambda(x).
\]
The operator $A_\lambda$ is monotone and $\tfrac{1}{\lambda}$-Lipschitz continuous (in fact, cocoercive with constant $\lambda$), and $A_\lambda$ converges to $A$ in the graph sense as $\lambda \downarrow 0$. In the particular case where $A = \partial f$ for a proper, lower semicontinuous, convex function $f$, the resolvent coincides with the proximal mapping $J_{\lambda A} = \operatorname{prox}_{\lambda f}$ and the Yosida approximation reduces to $A_\lambda = \nabla f_\lambda$, where $f_\lambda$ denotes the Moreau envelope of $f$. 

Analogously to the setting of optimization problems, we propose the following dynamic:  
\begin{equation*}
\ddot x(t)\;+\;\frac{\alpha}{t}\dot x(t)\;+\;\beta\,\frac{\mathrm{d}}{\mathrm{d}t}\!\Big[\delta(t)\,A_{\gamma(t)}(x(t))\Big]
\;+\;\Big(1+\frac{\beta}{t}\Big)\delta(t)\,A_{\gamma(t)}(x(t))\;=\;0, \quad (\text{HR-MMD})_{\alpha,\beta}
\end{equation*}
where $A_{\gamma(t)}$ is the Yosida approximation of $A$ with constant $\gamma(t)$.

Before going into the convergence analysis of $(\text{HR-MMD})_{\alpha,\beta}$, we need to make the following standing Assumptions.
\begin{Assumption}\label{ass:C}
The following conditions hold:
\begin{enumerate}[label=(\roman*)]
    \item $A : \mathcal{H}\to 2^{\mathcal{H}}$ be a maximally monotone operator.
    \item The set of zeros $\text{zer}A$ is nonempty. 
\end{enumerate}
\end{Assumption}
\begin{Assumption}\label{ass:D}
Let $\alpha > 1$, $\beta>0$, and $0<\sigma<\alpha-1$.  
Let $\delta,\gamma:[t_0,+\infty)\to(0,+\infty)$ be $C^1$ functions satisfying:

\begin{enumerate}[label=(\roman*)]
    \item $\displaystyle\lim_{t\to +\infty}\frac{\gamma(t)}{t^2\delta(t)}>\frac{1}{4(\alpha-\sigma-1)\sigma}.$
    \item 
    There exists $M>0$ such that for all large $t$,
    \[
    0 < \frac{t\dot{\delta}(t)}{\delta(t)} \le M.
    \]
    \item 
    $\displaystyle \frac{\vert\dot{\gamma}(t)\vert}{\gamma(t)} = O(1/t).$
\end{enumerate}
\end{Assumption}
\begin{remark}
A simple way to particularize the above assumptions is to choose $\delta$ and $\gamma$ as power functions of $t$ of the form
\[
\delta(t)=t^p,
\qquad
\gamma(t)=c\,t^{p+2},
\]
where $c>\frac{1}{4(\alpha-\sigma-1)\sigma}$, and $p>0$.

Indeed, we have
\[
\frac{\gamma(t)}{t^2\delta(t)}
=
\frac{c\,t^{p+2}}{t^2 t^p}
=c,
\]
hence assumption {\rm(i)} is satisfied provided that
\[
c>\frac{1}{4(\alpha-\sigma-1)\sigma}.
\]

Next,
\[
\dot\delta(t)=p\,t^{p-1},
\qquad\text{so}\qquad
\frac{t\dot\delta(t)}{\delta(t)}=p,
\]
and therefore assumption {\rm(ii)} holds.

Finally,
\[
\dot\gamma(t)=c(p+2)t^{p+1},
\qquad\text{so}\qquad
\frac{|\dot\gamma(t)|}{\gamma(t)}=\frac{p+2}{t},
\]
which shows that assumption {\rm(iii)} is also satisfied.

Therefore, the assumptions are fulfilled for the choice
\[
\delta(t)=t^p,
\qquad
\gamma(t)=c\,t^{p+2},
\]
with $p>0$, and
\[
c>\frac{1}{4(\alpha-\sigma-1)\sigma}.
\]
\end{remark}
The following lemma will be useful for the convergence analysis.
\begin{lemma}
\label{lemma:diag-majorant}
Let $\mathcal H$ be a real Hilbert space and let real numbers
\[
a<0,\qquad c<0,\qquad \Delta:=b^{2}-4ac<0
\]
be given.
For $x,y\in\mathcal H$ define
\[
P(x,y)=a\|x\|^{2}+b\langle x,y\rangle+c\|y\|^{2}.
\]
We have $P(x,y) \le 0 \quad \forall (x,y) \in \mathcal{H}^2$. \\
Furthermore, for two non-negative numbers $m,n$ the uniform inequality
\begin{equation}\label{ineq:goal}
P(x,y)\;\le\;-\,m\|x\|^{2}-n\|y\|^{2}\qquad\forall\,x,y\in\mathcal H
\end{equation}
holds {iff}
\begin{equation}\label{ineq:mn-conditions}
m\le -a,\quad n\le -c,\quad
\bigl(-m-a\bigr)\bigl(-n-c\bigr)\;\ge\;\frac{b^{2}}{4}\;.
\end{equation}
A particular choice for $(m,n)$ in which \eqref{ineq:goal} holds is
\begin{align*}
    0<m<\frac{\Delta}{4c}, \quad n=\frac{\Delta-4mc}{4(m+a)}.
\end{align*}
Note that with this choice, we have $n>0$.
\end{lemma}
\begin{proof}
Let us introduce the bounded self-adjoint operator
\[
A=
\begin{pmatrix}
a\,I & \tfrac12\,b\,I \\
\tfrac12\,b\,I & c\,I
\end{pmatrix}\;:\;\mathcal H^{2}\longrightarrow\mathcal H^{2},
\]
where $I$ is the identity on $\mathcal{H}$. We have 
\(
P(x,y)=\bigl\langle (x,y),A(x,y)\bigr\rangle_{\mathcal H^{2}}.
\)
Because $a<0,\;c<0$ and $\Delta<0$, the $2\times2$ coefficient matrix $\begin{pmatrix}
a & \tfrac12\,b\ \\
\tfrac12\,b\ & c\
\end{pmatrix}$ is
negative-definite; hence $A\prec0$ or in other words $P(x,y) \le 0 \quad \forall (x,y)\in \mathcal{H}^2$. \\[2mm]
For $m,n\ge0$ put
\[
D:=-\operatorname{diag}(mI,nI)=
\begin{pmatrix}-mI&0\\0&-nI\end{pmatrix}.
\]
Inequality \eqref{ineq:goal} is equivalent to the inequality
\begin{equation}\label{ineq:loewner}
D-A\succeq0 .
\end{equation}
\medskip
Explicitly,
\[
D-A=
\begin{pmatrix}
-\,m\,I & 0\\
0 & -\,n\,I
\end{pmatrix}
-
\begin{pmatrix}
a\,I & \tfrac12\,b\,I\\
\tfrac12\,b\,I & c\,I
\end{pmatrix}
=
\begin{pmatrix}
-(m+a)\,I & -\tfrac12\,b\,I\\[4pt]
-\tfrac12\,b\,I & -(n+c)\,I
\end{pmatrix}.
\]
Because each block is a multiple of $I$, $D-A\succeq0$ on $\mathcal H^{2}$
{iff} the scalar matrix
\[
B:=
\begin{pmatrix}
p & -\tfrac12\,b\\[4pt]
-\tfrac12\,b & q
\end{pmatrix},
\qquad
p:=-m-a,\; q:=-n-c,
\]
is positive-semidefinite on $\mathbb R^{2}$.\\
\medskip
For a symmetric $2\times2$ matrix, $B\succeq0$ is equivalent to
\[
p\ge0,\qquad q\ge0,\qquad pq-\frac{b^{2}}{4}\ge0.
\]
Re-expressing $p,q$ in terms of $m,n$ yields exactly the conditions
\eqref{ineq:mn-conditions}.
\end{proof}

For the convergence analysis we employ the Lyapunov function 
\begin{align*}
    \mathcal{V}_{A}(t)=\frac{1}{2}\Big\Vert \sigma(x(t)-x^*)+t\dot{x}(t)+\beta t\delta(t)\,A_{\gamma(t)}(x(t)) \Big\Vert^2 +\eta\Vert x(t)-x^*\Vert^2, 
\end{align*}
where $\sigma>0$ is the constant appearing in Assumption \ref{ass:C} and $\eta>0$ is to be chosen later.
In contrast to optimization settings, this Lyapunov function does not involve any quantity associated with an objective function. 

Taking the derivative of $\mathcal{V}_A$ gives 
\begin{align}
\label{derivative-lyapunov}
    \dot{\mathcal{V}}_{A}(t)=&\Big \langle \sigma(x(t)-x^*)+t\dot{x}(t)+\beta t\delta(t)\,A_{\gamma(t)}(x(t)), (\sigma+1)\dot{x}(t)+t\ddot{x}(t)+\beta\frac{\mathrm{d}}{\mathrm{d}t}\Big[t\delta(t)\,A_{\gamma(t)}(x(t))\Big] \Big\rangle \notag\\
    &+2\eta\langle x(t)-x^*,\dot{x}(t) \rangle. 
\end{align}
From $(\text{HR-MMD})_{\alpha,\beta}$, we have 
\begin{align*}
    t\ddot{x}(t)=-\alpha\dot{x}(t)-\beta t\frac{\mathrm{d}}{\mathrm{d}t}\Big[\delta(t)\,A_{\gamma(t)}(x(t))\Big]-(t+\beta)\delta(t)\,A_{\gamma(t)}(x(t)).
\end{align*}
Plugging this in \eqref{derivative-lyapunov} and simplifications give
\begin{align*}
\dot{\mathcal{V}}_{A}(t)=&\Big \langle \sigma(x(t)-x^*)+t\dot{x}(t)+\beta t\delta(t)\,A_{\gamma(t)}(x(t)), (\sigma+1-\alpha)\dot{x}(t)-t\delta(t)\,A_{\gamma(t)}(x(t)) \Big\rangle \\
    &+2\eta\langle x(t)-x^*,\dot{x}(t) \rangle. 
\end{align*}
With further algebraic manipulations, we obtain 
\begin{align*}
    \dot{\mathcal{V}}_{A}(t)=&\Big[\sigma(\sigma+1-\alpha)+2\eta\Big]\langle x(t)-x^*,\dot{x}(t) \rangle +t(\sigma+1-\alpha)\Vert \dot{x}(t) \Vert^2 \\
    &+\Big[\beta(\sigma+1-\alpha)-t\Big]t\delta(t)\Big\langle A_{\gamma(t)}(x(t)),\dot{x}(t) \Big\rangle -\sigma t\delta(t) \Big \langle A_{\gamma(t)}(x(t)),x(t)-x^{*} \Big \rangle \\
    &-\beta t^2\delta(t)^2\Big\Vert A_{\gamma(t)}(x(t)) \Big \Vert^2.
\end{align*}
Since $A_{\gamma(t)}$ is $\gamma(t)-$cocoercive, 
\begin{align*}
\Big \langle A_{\gamma(t)}(x(t)),x(t)-x^{*} \Big \rangle \geq \gamma(t) \Big \Vert A_{\gamma(t)}(x(t)) \Big \Vert^2. 
\end{align*}
In turn, 
\begin{align*}
    \dot{\mathcal{V}}_{A}(t) &\le \Big[\sigma(\sigma+1-\alpha)+2\eta\Big]\langle x(t)-x^*,\dot{x}(t) \rangle +t(\sigma+1-\alpha)\Vert \dot{x}(t) \Vert^2 \\
    &\hspace{5mm}+\Big[\beta(\sigma+1-\alpha)-t\Big]t\delta(t)\Big\langle A_{\gamma(t)}(x(t)),\dot{x}(t) \Big\rangle  -\Big({\sigma t\gamma(t)\delta(t)}+{\beta t^2\delta(t)^2}\Big)\Big\Vert A_{\gamma(t)}(x(t)) \Big \Vert^2 \\
    &= \Big[\sigma(\sigma+1-\alpha)+2\eta\Big]\langle x(t)-x^*,\dot{x}(t) \rangle + \Omega(t),
\end{align*}
where we have introduced the new function $\Omega(t)$ for ease of notation. 

We choose $\eta$ in the Lyapunov function such that $\sigma(\sigma+1-\alpha)+2\eta=0$ and define $B:=\alpha-\sigma-1>0$. In turn, $\dot{\mathcal{V}}_{A}(t) = \Omega(t)$. 

We will see that $\Omega(t) \le 0$ for sufficiently large $t$. To facilitate the process, let us introduce some notations
\begin{align*}
    &a(t):=t(\sigma+1-\alpha),\\
    &b(t):=\Big[\beta(\sigma+1-\alpha)-t\Big]t\delta(t),\\
    &c(t):=-\Big({\sigma t\gamma(t)\delta(t)}+{\beta t^2\delta(t)^2}\Big). 
\end{align*}
By Lemma \ref{lemma:diag-majorant} and that $a(t)<0$ and $c(t)<0$, in order for $\Omega(t) \le 0$, it is sufficient to have $\Delta(t)=b(t)^2-4a(t)c(t) < 0$. Indeed, after elementary computations, we arrive at
\begin{align*}
    \Delta(t)={t^4\delta(t)^2}\Big[\Big(1-\frac{\beta B}{t}\Big)^2-4B\sigma \frac{\gamma(t)}{t^2\delta(t)}\Big].
\end{align*}
Since $\lim_{t\to +\infty}\frac{\gamma(t)}{t^2\delta(t)}>\frac{1}{4B\sigma}$, we obtain that $\Delta(t) < 0$ as long as $t$ is large enough. 

Additionally, we can also compute 
\begin{align*}
    \frac{\Delta(t)}{4tc(t)}=\frac{4B\sigma\frac{\gamma(t)}{t^2\delta(t)}-(1-\frac{\beta B}{t})^2}{4\sigma\frac{\gamma(t)}{t^2\delta(t)}+4\frac{\beta}{t}}.
\end{align*}
Therefore 
\begin{align*}
\lim_{t \to +\infty} \frac{\Delta(t)}{4tc(t)} =\frac{4B\sigma L-1}{4\sigma L}=B-\frac{1}{4\sigma L}>0,  
\end{align*}
where $L:=\lim_{t\to +\infty}\frac{\gamma(t)}{t^2\delta(t)} >0.$ 

If we choose a $\varepsilon$ such that $0<\varepsilon<B-\frac{1}{4\sigma L}$ and set $m(t):=\varepsilon t$, we have for sufficiently large $t$,
\begin{align*}
    m(t)<\frac{\Delta(t)}{4c(t)}.
\end{align*}
We now compute $n(t)$ according to Lemma \ref{lemma:diag-majorant} defined as 
\begin{align*}
    n(t):= \frac{\Delta(t)-4m(t)c(t)}{4(m(t)+a(t))} 
    =t\,\sigma\,\delta(t)\gamma(t)
+\frac{t\,\delta(t)^{2}\Big((t-\beta B)^{2}+4\varepsilon\beta\,t\Big)}{4(\varepsilon-B)}.
\end{align*}
It should be noted that $n(t) > 0$ for all $t \geq t_0$. 

Applying Lemma \ref{lemma:diag-majorant} yields for sufficiently large $t$, say for $t \geq t_1\geq t_0$, that
\begin{align}
    \label{ine:derivative of Lyapunov}\dot{\mathcal{V}}_{A}(t) \le -\varepsilon t\Vert \dot{x}(t)\Vert^2 -n(t)\Big\Vert A_{\gamma(t)}(x(t)) \Big \Vert^2.
\end{align}
Integrating Inequality \eqref{ine:derivative of Lyapunov} gives 
\begin{align*}
    \mathcal{V}_{A}(t)+\int_{t_1}^{t} \varepsilon \tau\Vert \dot{x}(\tau) \Vert^2 d\tau +\int_{t_1}^{t} n(\tau)\Big\Vert A_{\gamma(\tau)}(x(\tau)) \Big \Vert^2 \mathrm{d}\tau \le \mathcal{V}_{A}(t_1).
\end{align*}
Combining this inequality with the definition of $\mathcal{V}_{A}$, we obtain the following immediate implications: 
\begin{align}
&\sup_{t\geq t_0}\Vert x(t) \Vert < +\infty \\
&\sup_{t\geq t_0}\Big\Vert t\dot{x}(t)+\beta t\delta(t)A_{\gamma(t)}(x(t)) \Big\Vert <+\infty. \label{property: sum of velo and A}\\
&\int_{t_0}^{+\infty} t\Vert \dot{x}(t) \Vert^2 \mathrm{d}t < +\infty. \label{velocity convergence:ingre1} \\
&\int_{t_1}^{+\infty} t^3\delta(t)^2\Big\Vert A_{\gamma(t)}(x(t)) \Big \Vert^2 \mathrm{d}t <+\infty, \label{velocity convergence:ingre3}
\end{align}
where~\eqref{velocity convergence:ingre3} is obtained by noticing that 
\begin{align*}
    \lim_{t \to +\infty} \frac{n(t)}{t^3\delta(t)^2}&=\lim_{t \to +\infty} \sigma \frac{\gamma(t)}{t^2\delta(t)}+\frac{(t-\beta B)^2+4\varepsilon \beta t}{4(\varepsilon-B)t^2} \\
    &=\sigma L+\frac{1}{4(\varepsilon-B)} >0 \quad (\text{since } 0<\varepsilon<B-\frac{1}{4\sigma L}).
\end{align*}
The inequality~\eqref{property: sum of velo and A} can actually be strengthened to 
\begin{align}
     &\sup_{t\geq t_0}t\Vert \dot{x}(t)\Vert <+\infty \text{ or equivalently } \Vert \dot{x}(t) \Vert={O}({1}/{t}). \label{convergence: velocity} \\
    &t\delta(t)\Big\Vert A_{\gamma(t)}(x(t)) \Big\Vert = {O}({1}/{t}) \text{ or equivalently } \Big\Vert A_{\gamma(t)}(x(t)) \Big\Vert= O\Big(\frac{1}{t^2\delta(t)}\Big). \label{BigO_gradient}
\end{align}
Indeed, by virtue of the $\frac{1}{\gamma(t)}-$Lipschitz continuity of $A_{\gamma(t)}$, we have for all $t \geq t_0$
\begin{align*}
  t\delta(t)\Big\Vert A_{\gamma(t)}(x(t)) \Big\Vert&=t\delta(t)\Big\Vert A_{\gamma(t)}(x(t))-A_{\gamma(t)}(x^*) \Big\Vert \\
  &\le  \frac{t\delta(t)}{\gamma(t)}\Vert x(t)-x^* \Vert \\
  &\le \frac{t\delta(t)}{\gamma(t)}\sup_{t \geq t_0} \Vert x(t)-x^* \Vert\\
  &={O}(1/t).
\end{align*}
The last equality holds since $x$ is bounded and $\lim_{t\to +\infty}\frac{\gamma(t)}{t^2\delta(t)}>\frac{1}{4B\sigma}$. 

We now turn to deriving improved convergence rates for $\Vert \dot{x}(t) \Vert $ and $\Vert A_{\gamma(t)}(x(t)) \Vert $. 
The following lemma will be useful.

\begin{lemma}\label{lem:derivative of moreau}
Let $\mathcal H$ be a real Hilbert space and let $A:\mathcal H\rightrightarrows\mathcal H$ be maximally monotone.
For $\gamma>0$ denote the resolvent $J_{\gamma A}=(I+\gamma A)^{-1}$ and the Yosida approximation $A_\gamma=\frac{1}{\gamma}(I-J_{\gamma A})$.
Let $\delta,\gamma:[t_0,+\infty)\to(0,+\infty)$ and $x:[t_0,+\infty)\to\mathcal H$ be $C^1$. Then, for any $x^{\star}\in\mathrm{zer}\,A$ and all $t\ge t_0$,
\[
\Big\Vert \tfrac{\mathrm{d}}{\mathrm{d}t}\big[\delta(t)A_{\gamma(t)}(x(t))\big] \Big\Vert
\le \frac{\delta(t)}{\gamma(t)}\Vert \dot{x}(t) \Vert
+ \Big[\frac{2\delta(t)\vert \dot{\gamma}(t)\vert}{\gamma(t)^2}
+ \frac{\vert \dot{\delta}(t)\vert}{\gamma(t)}\Big]\Vert x(t)-x^* \Vert .
\]
\end{lemma}

\begin{proof}
We use standard facts: for each $\gamma>0$, $J_{\gamma A}$ is firmly nonexpansive, and $A_\gamma$ is $1/\gamma$-Lipschitz and $\gamma$-cocoercive, i.e.,
\begin{equation}\label{eq:Lip-coco}
\|A_\gamma x-A_\gamma y\|\le \frac{1}{\gamma}\|x-y\|, \qquad
\langle A_\gamma x-A_\gamma y,\,x-y\rangle \ge \gamma\|A_\gamma x-A_\gamma y\|^2 .
\end{equation}

\textbf{Step 1: A basic bound via a zero of $A$.}
Fix $x^{\star}\in\mathrm{zer}\,A$ (so $J_{\gamma A}x^{\star}=x^{\star}$ and $A_\gamma(x^{\star})=0$). Apply \eqref{eq:Lip-coco} with $y=x^{\star}$ and use Cauchy--Schwarz inequality:
\begin{equation}
\gamma\|A_\gamma(x)\|^2 \le \langle A_\gamma(x),x-x^{\star}\rangle
\le \|A_\gamma(x)\|\,\|x-x^{\star}\| \;\Rightarrow\;
{\;\|A_\gamma(x)\|\le \frac{1}{\gamma}\|x-x^{\star}\|.\;}
\label{eq:basic-bound}
\end{equation}

\textbf{Step 2: Resolvent identity and parameter perturbation.}
We first show the resolvent identity: for any $\alpha,\beta>0$ and $x\in\mathcal H$,
\begin{equation}\label{eq:RI}
{\quad
J_{\alpha A}
= J_{\beta A}\!\left(\tfrac{\beta}{\alpha}I+\Big(1-\tfrac{\beta}{\alpha}\Big)J_{\alpha A}\right).
\quad}
\end{equation}
Indeed, let $u=J_{\alpha A}x$, so $(x-u)/\alpha\in Au$. Multiplying by $\beta$ gives
$(\beta/\alpha)(x-u)\in \beta Au$, hence
\[
u=(I+\beta A)^{-1}\!\Big(\tfrac{\beta}{\alpha}x+\Big(1-\tfrac{\beta}{\alpha}\Big)u\Big)
=J_{\beta A}\!\Big(\tfrac{\beta}{\alpha}x+\Big(1-\tfrac{\beta}{\alpha}\Big)u\Big),
\]
which is \eqref{eq:RI}.

From \eqref{eq:RI} we derive the parameter perturbation bound: for $\alpha,\beta>0$ and all $x$,
\begin{equation}\label{eq:PP}
{\quad
\|A_\beta(x)-A_\alpha(x)\|
\le \frac{2|\beta-\alpha|}{\alpha}\,\|A_\beta(x)\| .
\quad}
\end{equation}
To see this, set $u=J_{\alpha A}x$ and $v=J_{\beta A}x$. Using \eqref{eq:RI} with $\alpha,\beta$ swapped gives
\[
v=J_{\alpha A}\!\Big(\tfrac{\alpha}{\beta}x+\Big(1-\tfrac{\alpha}{\beta}\Big)v\Big).
\]
By nonexpansiveness of $J_{\alpha A}$,
\begin{equation}\label{eq:v-u}
\|v-u\|
\le \left\|\Big(\tfrac{\alpha}{\beta}-1\Big)x+\Big(1-\tfrac{\alpha}{\beta}\Big)v\right\|
= \frac{|\alpha-\beta|}{\beta}\,\|x-v\|.
\end{equation}
Now
\[
A_\beta(x)-A_\alpha(x)
= \frac{x-v}{\beta}-\frac{x-u}{\alpha}
= \Big(\frac{1}{\beta}-\frac{1}{\alpha}\Big)(x-v) + \frac{1}{\alpha}(v-u),
\]
whence, using \eqref{eq:v-u},
\[
\|A_\beta(x)-A_\alpha(x)\|
\le \frac{|\alpha-\beta|}{\alpha\beta}\|x-v\|
+ \frac{1}{\alpha}\cdot \frac{|\alpha-\beta|}{\beta}\|x-v\|
= \frac{2|\alpha-\beta|}{\alpha\beta}\|x-v\|.
\]
Since $\|x-v\|/\beta=\|A_\beta(x)\|$, we obtain \eqref{eq:PP}.

\textbf{Step 3: Difference quotient and limit.}
Define $g(t):=\delta(t)A_{\gamma(t)}(x(t))$. For small $h\neq0$,
\[
\begin{aligned}
\|g(t+h)-g(t)\|
&= \|\delta(t+h)A_{\gamma(t+h)}(x(t+h))-\delta(t)A_{\gamma(t)}(x(t))\|\\
&\le \underbrace{\delta(t)\|A_{\gamma(t)}(x(t+h))-A_{\gamma(t)}(x(t))\|}_{\mathrm{(I)}}\\
&\quad + \underbrace{\delta(t)\|A_{\gamma(t+h)}(x(t+h))-A_{\gamma(t)}(x(t+h))\|}_{\mathrm{(II)}}\\
&\quad + \underbrace{|\delta(t+h)-\delta(t)|\,\|A_{\gamma(t+h)}(x(t+h))\|}_{\mathrm{(III)}}.
\end{aligned}
\]
Using the $1/\gamma$-Lipschitzness in $x$,
\[
\mathrm{(I)} \le \frac{\delta(t)}{\gamma(t)}\,\|x(t+h)-x(t)\|.
\]
Using \eqref{eq:PP} with $\alpha=\gamma(t)$ and $\beta=\gamma(t+h)$,
\[
\mathrm{(II)} \le \frac{2\delta(t)|\gamma(t+h)-\gamma(t)|}{\gamma(t)}\,\|A_{\gamma(t+h)}(x(t+h))\|.
\]
Applying \eqref{eq:basic-bound} at $(x,\gamma)=(x(t+h),\gamma(t+h))$ yields
\[
\mathrm{(II)} \le \frac{2\delta(t)|\gamma(t+h)-\gamma(t)|}{\gamma(t)\gamma(t+h)}\,\|x(t+h)-x^*\|.
\]
Similarly, by \eqref{eq:basic-bound},
\[
\mathrm{(III)} \le \frac{|\delta(t+h)-\delta(t)|}{\gamma(t+h)}\,\|x(t+h)-x^*\|.
\]

Divide by $|h|$ and let $h\to0$. Since $x,\gamma,\delta$ are $C^1$ and $(x,\gamma)\mapsto A_\gamma(x)$ is continuous, we obtain
\[
\Big\Vert \tfrac{\mathrm{d}}{\mathrm{d}t}\big[\delta(t)A_{\gamma(t)}(x(t))\big] \Big\Vert
\le \frac{\delta(t)}{\gamma(t)}\Vert \dot{x}(t) \Vert
+ \Big[\frac{2\delta(t)\vert \dot{\gamma}(t)\vert}{\gamma(t)^2}
+ \frac{\vert \dot{\delta}(t)\vert}{\gamma(t)}\Big]\Vert x(t)-x^* \Vert,
\]
as claimed.
\end{proof}
By Lemma \ref{lem:derivative of moreau} and Assumption \ref{ass:D}, we obtain
\begin{align*}
\Big \Vert \frac{\mathrm{d}}{\mathrm{d}t}\Big[\delta(t)A_{\gamma(t)}(x(t))\Big] \Big \Vert ={O}\Big(\frac{1}{t^3}\Big). 
\end{align*}
A consequence of this fact is that 
\begin{align}
\label{velocity convergence:ingre2}
\int_{t_0}^{+\infty} t^3\Big \Vert \frac{\mathrm{d}}{\mathrm{d}t}\Big[\delta(t)A_{\gamma(t)}(x(t))\Big] \Big \Vert^2\mathrm{d}t <+\infty.
\end{align}
Further, we have by the definition of the dynamical system, 
\begin{align*}
    t^3\Vert \ddot{x}(t)\Vert^2 &=t^3 \Big\Vert -\frac{\alpha}{t}\dot{x}(t)-\beta\frac{\mathrm{d}}{\mathrm{d}t}\Big[\delta(t)A_{\gamma(t)}(x(t))\Big]-\Big(1+\frac{\beta}{t}\Big)\delta(t)A_{\gamma(t)}(x(t))\Big\Vert^2 \\
    &\le 3\alpha^2t\Vert \dot{x}(t) \Vert^2+3\beta^2t^3\Big \Vert \frac{\mathrm{d}}{\mathrm{d}t}\Big[\delta(t)A_{\gamma(t)}(x(t))\Big] \Big \Vert^2+3t(t+\beta)^2\delta(t)^2\Big \Vert A_{\gamma(t)}(x(t)) \Big \Vert^2.
\end{align*}
Now, from \eqref{velocity convergence:ingre1}, \eqref{velocity convergence:ingre2}, and \eqref{velocity convergence:ingre3}, we deduce that the right-hand side of the above inequality belongs to $L^1([t_0,+\infty))$, and hence 
\begin{align}
\label{acceleraton:integral}
t^3\Vert \ddot{x}(t)\Vert^2 \in L^1([t_0,+\infty)).
\end{align}
We are now ready to prove that the rate of convergence for $\Vert \dot{x}(t)\Vert$ is actually $\Vert \dot{x}(t)\Vert=o\big(\frac{1}{t}\big)$. \\
Let us recall from \eqref{velocity convergence:ingre1} that $\int_{t_0}^{+\infty} t\Vert \dot{x}(t) \Vert^2 \mathrm{d}t < +\infty$. So $\liminf_{t\rightarrow+\infty}t\Vert \dot{x}(t) \Vert =0$. Hence, we only need to prove that $\lim_{t \rightarrow+\infty}t\Vert \dot{x}(t) \Vert$ exists. Indeed, 
\begin{align*}
\frac{\mathrm{d}}{\mathrm{d}t} t^2\Vert\dot{x}(t)\Vert^2 &=2t\Vert \dot{x}(t) \Vert^2+2t^2\langle \dot{x}(t),\ddot{x}(t)\rangle \\
&\le3t\Vert \dot{x}(t) \Vert^2 + t^3\Vert \ddot{x}(t) \Vert^2.
\end{align*}
From \eqref{velocity convergence:ingre1} and \eqref{acceleraton:integral}, we deduce that the right-hand side of the above inequality also belongs to $L^1([t_0,+\infty))$. This implies that $\lim_{t \rightarrow+\infty}t^2\Vert \dot{x}(t) \Vert^2$ exists, and so does $\lim_{t \rightarrow+\infty}t\Vert \dot{x}(t) \Vert$. 

We now turn to improving the rate $\Big\Vert A_{\gamma(t)}(x(t)) \Big\Vert={O}\Big(\frac{1}{t^2\delta(t)}\Big)$ to $\Big\Vert A_{\gamma(t)}(x(t)) \Big\Vert=o\Big(\frac{1}{t^2\delta(t)}\Big)$. 

From \eqref{velocity convergence:ingre3}, we have 
\begin{align*}
    \liminf_{t \rightarrow+\infty} t^2\delta(t)\Big\Vert A_{\gamma(t)}(x(t)) \Big\Vert=0.
\end{align*}
Therefore, it is sufficient to show that $\lim_{t \rightarrow+\infty} t^2\delta(t)\Big\Vert A_{\gamma(t)}(x(t)) \Big\Vert$ exists. To this end, let us set
\begin{align*}
    \xi(t)=\Vert t^2\delta(t)A_{\gamma(t)}(x(t)) \Vert^2.
\end{align*}
We have the following chain of implications,
\begin{align*}
    \frac{\mathrm{d}}{\mathrm{d}t}\xi(t) 
    &=2\Big\langle t^2\delta(t)A_{\gamma(t)}(x(t)), 2t\delta(t)A_{\gamma(t)}(x(t))+t^2\frac{\mathrm{d}}{\mathrm{d}t}\delta(t)A_{\gamma(t)}(x(t)) \Big\rangle \\
    &=4t^3\delta(t)^2\Vert A_{\gamma(t)}(x(t)) \Vert^2+2t^4\delta(t)\Big\langle A_{\gamma(t)} (x(t)),\frac{\mathrm{d}}{\mathrm{d}t}\delta(t)A_{\gamma(t)}(x(t)) \Big\rangle\\
    &\le 4t^3\delta(t)^2\Vert A_{\gamma(t)}(x(t)) \Vert^2+2t^4\delta(t)\Vert A_{\gamma(t)} (x(t))\Vert \Big\Vert\frac{\mathrm{d}}{\mathrm{d}t}\delta(t)A_{\gamma(t)}(x(t)) \Big\Vert \\ 
    &\le  4t^3\delta(t)^2\Vert A_{\gamma(t)}(x(t)) \Vert^2+\frac{2t^4\delta(t)^2}{\gamma(t)}\Vert A_{\gamma(t)}(x(t)) \Vert \Vert \dot{x}(t) \Vert\\
    &\hspace{5mm} +\Big(\frac{4t^4\delta(t)^2\vert \dot{\gamma}(t)\vert}{\gamma(t)}+2t^4\delta(t)\vert \dot{\delta}(t)\vert\Big)\Vert A_{\gamma(t)}(x(t)) \Vert^2,
\end{align*}
where the last inequality follows from Lemma \ref{lem:derivative of moreau}. 

From \ref{velocity convergence:ingre3}, we have that
\begin{align*}
4t^3\delta(t)^2\Vert A_{\gamma(t)}(x(t)) \Vert^2 \in L^1([t_0,+\infty)).   
\end{align*}
In addition, from Assumption \ref{ass:D}\textit{(ii)}, \textit{(iii)}, we obtain the rate 
\begin{align*}
    \Big(\frac{4t^4\delta(t)^2\vert \dot{\gamma}(t)\vert}{\gamma(t)}+2t^4\delta(t)\vert \dot{\delta}(t)\vert\Big)\Vert A_{\gamma(t)}(x(t)) \Vert^2 = O\Big(t^3\delta(t)^2\Vert A_{\gamma(t)}(x(t)) \Vert^2\Big),
\end{align*}
which belongs to $L^1([t_0,+\infty))$ due to \eqref{velocity convergence:ingre3}.

Moreover,  
\begin{align*}
\frac{2t^4\delta(t)^2}{\gamma(t)}\Vert A_{\gamma(t)}(x(t)) \Vert \Vert \dot{x}(t) \Vert &=\frac{2t^4\delta(t)^2}{\gamma(t)}\sqrt{\frac{\gamma(t)}{t}}\Vert A_{\gamma(t)}(x(t)) \Vert \sqrt{\frac{t}{\gamma(t)}}\Vert \dot{x}(t) \Vert \\
&\le \frac{t^4\delta(t)^2}{\gamma(t)}\Big(\frac{\gamma(t)}{t}\Vert A_{\gamma(t)}(x(t)) \Vert^2+\frac{t}{\gamma(t)}\Vert \dot{x}(t) \Vert^2\Big) \\
&=t^3\delta(t)^2\Vert A_{\gamma(t)}(x(t)) \Vert^2+\frac{t^5\delta(t)^2}{\gamma(t)^2}\Vert \dot{x}(t) \Vert^2 \\
&=t^3\delta(t)^2\Vert A_{\gamma(t)}(x(t)) \Vert^2+O(t\Vert \dot{x}(t) \Vert^2) \text{ (thanks to Assumption \ref{ass:D}\textit{(i)})}
\end{align*}
This, combined with \eqref{velocity convergence:ingre1} and \eqref{velocity convergence:ingre3}, implies 
\begin{align*}
 \frac{2t^4\delta(t)^2}{\gamma(t)}\Vert A_{\gamma(t)}(x(t)) \Vert \Vert \dot{x}(t) \Vert  \in L^1([t_0,+\infty)).   
\end{align*}
As a result, 
\begin{align*}
    \frac{\mathrm{d}}{\mathrm{d}t}\xi(t) 
    &\le 4t^3\delta(t)^2\Vert A_{\gamma(t)}(x(t)) \Vert^2+\frac{2t^4\delta(t)^2}{\gamma(t)}\Vert A_{\gamma(t)}(x(t)) \Vert \Vert \dot{x}(t) \Vert+\\&\hskip1cm+\Big(\frac{4t^4\delta(t)^2\vert \dot{\gamma}(t)\vert}{\gamma(t)}+2t^4\delta(t)\vert \dot{\delta}(t)\vert\Big)\Vert A_{\gamma(t)}(x(t)) \Vert^2 
    \in L^1([t_0,+\infty)). 
\end{align*}
Hence, we can classically conclude that $\lim_{t \rightarrow+\infty} \xi(t)$ exists, and so does $\lim_{t \rightarrow+\infty} t^2\delta(t)\Big\Vert A_{\gamma(t)}(x(t)) \Big\Vert$. \\
Following the same lines of proof in the optimization setting, one can show the weak convergence of $x(t)$ as $t \to +\infty$ to a zero of $A$. We, therefore, omit the details.
Let us summarize what has been shown in the following theorem.
\begin{theorem}
Under the standing Assumptions \ref{ass:C} and \ref{ass:D}, and for any solution trajectory $x:[t_0,+\infty) \to \mathcal{H}$ to $(\text{HR-MMD})_{\alpha,\beta}$, we have 
\begin{enumerate}[label=(\roman*)]
        \item $\displaystyle \Vert A_{\gamma(t)}(x(t))\Vert=o\Big(\frac{1}{t^2\delta(t)}\Big).$
        \item $\Vert \dot{x}(t) \Vert = o(\frac{1}{t})$.
        \item $\displaystyle \int_{t_0}^{+\infty}t\Vert \dot{x}(t)\Vert^2\mathrm{d}t < +\infty$. 
        \item $\displaystyle \int_{t_0}^{+\infty}t^3\delta(t)^2\Vert A_{\gamma(t)}(x(t))\Vert^2 \mathrm{d}t <+\infty$.
        \item $x(t)$ converges weakly as $t \to +\infty$, and its limit is a zero of $A$.
    \end{enumerate}
\end{theorem}

\section{Discussion on optimization vs. maximally monotone operators}
The analyses for nonsmooth convex minimization and for maximally monotone operators are developed in parallel since they rely on the same regularization mechanism. In both cases, the nonsmooth object is replaced by a single-valued and Lipschitz continuous approximation: the gradient of the Moreau envelope $\nabla f_{\gamma(t)}$ in the optimization setting, and the Yosida approximation $A_{\gamma(t)}$ in the operator setting. This makes it possible to define a high-resolution inertial dynamic with Hessian-driven damping and time-rescaled forcing without assuming second-order smoothness.

When $A = \partial f$, the two frameworks coincide exactly. Indeed, the Yosida approximation reduces to $A_\gamma = \nabla f_\gamma$, and the operator-valued dynamic (HR-MMD)$_{\alpha,\beta}$ becomes identical to (NS-HR)$_{\alpha,\beta}$. Although the dynamics are the same, the viewpoints are slightly different. In the optimization case, the analysis naturally focuses on the decay of function values $f_{\gamma(t)}(x(t)) - f^\star$, while in the operator framework the emphasis is on the decay of the residual $\|A_{\gamma(t)}(x(t))\|$, which is the appropriate notion of optimality for general monotone inclusions.

This comparison shows that the proposed dynamics provide a unified framework. Nonsmooth convex minimization appears as a particular case of the monotone inclusion problem, where additional convex-analytic structure allows for sharper Lyapunov estimates. At the same time, the operator formulation clarifies that the acceleration mechanisms studied here are not specific to objective minimization, but extend naturally to general monotone operators.

It is natural to ask how the standing assumptions compare between the two frameworks considered in this paper:
the convex optimization case (Assumption~\ref{ass:B}) and the maximally monotone inclusion case
(Assumption~\ref{ass:D}). Although the dynamics have the same structure, the sufficient conditions on the
time-dependent parameters $(\delta,\gamma)$ are not the same.

\smallskip
\noindent
\textbf{Optimization case.}
In Assumption~\ref{ass:B}, the key requirement is a lower bound on the {rate of variation} of the smoothing
parameter:
\[
\liminf_{t\to\infty}\frac{\dot\gamma(t)}{t\,\delta(t)} > 0.
\]
In addition, the growth of $\delta$ is controlled by $\alpha$ through
Assumption~\ref{ass:B}\textit{(iv)}, namely
\[
0\le\frac{t\dot\delta(t)}{\delta(t)}\le \alpha-3-\zeta,
\]
so in this setting $\delta$ cannot grow too fast. In particular, for polynomial schedules
\[
\delta(t)=t^p,
\]
this condition becomes simply
\[
\frac{t\dot\delta(t)}{\delta(t)}=p,
\]
so one must have
\[
0\le p<\alpha-3.
\]
Therefore, in the optimization case, the time-rescaling $\delta$ is explicitly coupled with the damping
parameter $\alpha$.

\smallskip
\noindent
\textbf{Maximally monotone case.}
In Assumption~\ref{ass:D}(i), the role of $\dot\gamma(t)$ is replaced by a {size constraint} involving $\gamma(t)$ itself:
\[
\lim_{t\to\infty}\frac{\gamma(t)}{t^2\,\delta(t)}>\frac{1}{4(\alpha-\sigma-1)\sigma}.
\]
Thus, in the monotone case, what matters is that $\gamma(t)$ be sufficiently large relative to
$t^2\delta(t)$.
Compared with the optimization case, the growth control on $\delta$ is more flexible, since it only requires
\[
0<\frac{t\dot\delta(t)}{\delta(t)}\le M,
\]
without an explicit coupling with $\alpha$.

\smallskip
\noindent
\textbf{Incomparability.}
The two families of assumptions are, in general, incomparable. This is already visible for the polynomial
choice
\[
\delta(t)=t^p,\qquad \gamma(t)=c\,t^{p+2}.
\]
In this case,
\[
\frac{\dot\gamma(t)}{t\delta(t)}=c(p+2),\qquad
\frac{\gamma(t)}{t^2\delta(t)}=c,\qquad
\frac{t\dot\delta(t)}{\delta(t)}=p.
\]

\textit{Optimization may hold while the monotone condition fails.}
Assume
\[
0\le p<\alpha-3.
\]
Then Assumption~\ref{ass:B}\textit{(iv)} is satisfied, and moreover
\[
\frac{\dot\gamma(t)}{t\delta(t)}=c(p+2)>0
\]
for every $c>0$, so Assumption~\ref{ass:B}\textit{(i)} is also satisfied. However, if $c$ is chosen too small,
namely
\[
0<c<\frac{1}{4(\alpha-\sigma-1)\sigma},
\]
then
\[
\frac{\gamma(t)}{t^2\delta(t)}=c
\]
stays below the threshold required in Assumption~\ref{ass:D}\textit{(i)}. Hence the optimization
assumptions hold, whereas the monotone ones fail. The failure is therefore very explicit: $\gamma$
varies fast enough for the optimization proof, but its {magnitude} is still too small for the quadratic
estimate used in the monotone case.

\medskip
\noindent
\textit{Conversely, the monotone assumptions may hold while the optimization ones fail.}
Assume now that
\[
p\geq\alpha-3.
\]
Then
\[
\frac{t\dot\delta(t)}{\delta(t)}=p\geq\alpha-3,
\]
so Assumption~\ref{ass:B}\textit{(iv)} is violated: in the optimization setting, $\delta$ grows too fast relative
to $\alpha$. On the other hand, this does not prevent the monotone assumptions from being satisfied,
because Assumption~\ref{ass:D}\textit{(ii)} only asks for
\[
\frac{t\dot\delta(t)}{\delta(t)}\le M,
\]
which is true here by taking, for instance, any $M\ge p$. If in addition
\[
c>\frac{1}{4(\alpha-\sigma-1)\sigma},
\]
then Assumption~\ref{ass:D}\textit{(i)} is also satisfied. Thus the monotone framework admits some
polynomial schedules that are excluded in the optimization framework.

In summary, the incompatibility comes from two distinct sources. In the optimization case, one needs
a lower bound on the {variation rate} $\dot\gamma(t)/(t\delta(t))$ together with a growth restriction
on $\delta$ tied to $\alpha$. In the monotone case, one instead needs a lower bound on the {size}
$\gamma(t)/(t^2\delta(t))$, while the growth of $\delta$ is controlled only in a much looser way. 

\section{Discretization of the (NS-HR)$_{\alpha,\beta}$ dynamic}
In this section, we derive a discrete-time algorithm associated with the continuous dynamic \((\mathrm{NS\text{-}HR})_{\alpha,\beta}\). Our aim is to construct a discretization that reflects the second-order inertial structure of the system and, at the same time, remains tractable in the nonsmooth setting. As will be seen below, although the discretization leads to an implicit relation involving the regularized gradient, this relation can be resolved explicitly through a single proximal computation at each iteration. This yields a practical proximal scheme that can be regarded as the algorithmic counterpart of the continuous high-resolution model. 
\subsection{Second-order discretization}
We recall the continuous model (NS-HR)$_{\alpha,\beta}$
\begin{equation}
\ddot x(t)
+\frac{\alpha}{t}\dot x(t)
+\beta \frac{d}{dt}\!\Big[\delta(t)\nabla f_{\gamma(t)}(x(t))\Big]
+\Big(1+\frac{\beta}{t}\Big)\delta(t)\nabla f_{\gamma(t)}(x(t))
=0.
\end{equation}
Let $t_k = kh$, and define
\[
\delta_k := \delta(t_k),
\qquad
\gamma_k := \gamma(t_k),
\qquad
g_k := \nabla f_{\gamma_k}(x_k)
=\frac{1}{\gamma_k}\Big(x_k-\operatorname{prox}_{\gamma_k f}(x_k)\Big).
\]

Using the discretization
\[
\ddot x(t)
\approx
\frac{x_{k+1}-2x_k+x_{k-1}}{h^2},
\]
\[
\frac{\alpha}{t}\dot x(t)
\approx
\frac{\alpha}{kh}(x_k-x_{k-1}),
\]
\[
\frac{d}{dt}\!\Big[\delta\nabla f_\gamma(x)\Big](t)
\approx
\frac{\delta_k g_k-\delta_{k-1}g_{k-1}}{h},
\]
\[
\Big(1+\frac{\beta}{t}\Big)\delta(t)\nabla f_{\gamma(t)}(x(t))
\approx \frac{\beta}{kh}\delta_{k-1}g_{k-1} + \delta_{k+1}g_{k+1},
\]
we obtain
\begin{align}
\frac{x_{k+1}-2x_k+x_{k-1}}{h^2}
&+
\frac{\alpha}{kh}(x_k-x_{k-1})
+
\beta \frac{\delta_k g_k-\delta_{k-1}g_{k-1}}{h}
\nonumber \\
&+
\frac{\beta}{kh}\delta_{k-1}g_{k-1} + \delta_{k+1}g_{k+1}
=0.
\tag{D-NSHR}
\end{align}

Multiplying by $h^2$ and rearranging yields
\begin{equation}
x_{k+1}+s_{k+1}g_{k+1}
=
r_{k+1},
\tag{$\ast$}
\end{equation}
where
\[
s_{k+1}
:=
\delta_{k+1}h^2,
\]
\[
r_{k+1}
:=
2x_k-x_{k-1}
-\frac{\alpha h}{k}(x_k-x_{k-1})
-\beta h(\delta_k g_k-\delta_{k-1}g_{k-1})-\frac{\beta h}{k}\delta_{k-1}g_{k-1}.
\]

\subsection{Resolution via a single proximal step}

We start from the implicit equation obtained after discretization:
\begin{equation}
x_{k+1}
+
s_{k+1}\nabla f_{\gamma_{k+1}}(x_{k+1})
=
r_{k+1}.
\tag{$\ast$}
\end{equation}

Recall that the Moreau envelope gradient satisfies
\[
\nabla f_{\gamma}(x)
=
\frac{1}{\gamma}
\Big(x-\operatorname{prox}_{\gamma f}(x)\Big).
\]

Define
\[
u_{k+1}
:=
\operatorname{prox}_{\gamma_{k+1} f}(x_{k+1}).
\]

By definition of the proximal operator,
\[
x_{k+1}
=
u_{k+1}
+
\gamma_{k+1} g_{k+1},
\qquad
g_{k+1}\in\partial f(u_{k+1}),
\]
and
\[
g_{k+1}
=
\nabla f_{\gamma_{k+1}}(x_{k+1}).
\]

Plugging $x_{k+1}=u_{k+1}+\gamma_{k+1} g_{k+1}$ into ($\ast$) gives
\[
u_{k+1}
+
\gamma_{k+1} g_{k+1}
+
s_{k+1} g_{k+1}
=
r_{k+1}.
\]

Thus
\[
r_{k+1}-u_{k+1}
=
(\gamma_{k+1}+s_{k+1}) g_{k+1}.
\]

Since $g_{k+1}\in\partial f(u_{k+1})$, we obtain
\[
r_{k+1}-u_{k+1}
\in
(\gamma_{k+1}+s_{k+1}) \partial f(u_{k+1}).
\]

This is exactly the optimality condition of the proximal operator with
parameter
\[
\lambda_{k+1}
:=
\gamma_{k+1}+s_{k+1}.
\]

Therefore,
\[
u_{k+1}
=
\operatorname{prox}_{\lambda_{k+1} f}(r_{k+1}).
\]

Once $u_{k+1}$ is known,
\[
g_{k+1}
=
\frac{r_{k+1}-u_{k+1}}{\lambda_{k+1}},
\]
and
\[
x_{k+1}
=
u_{k+1}
+
\gamma_{k+1} g_{k+1}
=
\frac{s_{k+1}}{\lambda_{k+1}}u_{k+1}
+
\frac{\gamma_{k+1}}{\lambda_{k+1}}r_{k+1}.
\]

\medskip

\noindent

\subsection{Algorithm}

Given $x_0,x_1, \alpha>0, \beta>0, h>0$, and sequences $\{\delta_k\},\{\gamma_k\}$:

For $k\ge1$:

\begin{enumerate}
\item Compute
\[
p_k
=
\operatorname{prox}_{\gamma_k f}(x_k),
\qquad
g_k
=
\frac{x_k-p_k}{\gamma_k}.
\]

\item Form
\[
r_{k+1}
:=
2x_k-x_{k-1}
-\frac{\alpha h}{k}(x_k-x_{k-1})
-\beta h(\delta_k g_k-\delta_{k-1}g_{k-1})-\frac{\beta h}{k}\delta_{k-1}g_{k-1}.
\]

\item Compute
\[
s_{k+1}
=
\delta_{k+1}h^2,
\qquad
\lambda_{k+1}
=
\gamma_{k+1}+s_{k+1},
\]
\[
u_{k+1}
=
\operatorname{prox}_{\lambda_{k+1}f}(r_{k+1}),
\]
\[
x_{k+1}
=
\frac{s_{k+1}}{\lambda_{k+1}}u_{k+1}
+
\frac{\gamma_{k+1}}{\lambda_{k+1}}r_{k+1}.
\]
\end{enumerate}
We emphasize that the purpose of this section is only to derive a natural time-discrete counterpart of the continuous dynamic \((\mathrm{NS\text{-}HR})_{\alpha,\beta}\). Preliminary numerical experiments show that the algorithm converges under suitable parameter choices. However, a rigorous Lyapunov analysis of the resulting scheme is beyond the scope of the present paper and will be the topic of  future research.
\section{Numerical Experiments}

In this section, we investigate the numerical behavior of the proposed dynamic (NS-HR)$_{\alpha,\beta}$ and compare it with several related continuous-time models. Our goals are threefold: first, to study the influence of the parameter $\beta$ in our system; second, to examine the role of $\alpha$; and third, to compare our dynamic with several benchmark dynamics from the literature.

\subsection{Numerical setup}

We consider the convex nonsmooth objective function
\begin{equation}
\label{eq:num-f}
f(x)=\frac12\bigl(x_1^2+1000x_2^2\bigr)+\|x\|_1,
\qquad x=(x_1,x_2)\in\mathbb{R}^2.
\end{equation}
This function is separable, coercive, and admits the unique minimizer
\[
x^\star=(0,0),
\qquad
f^\star=f(x^\star)=0.
\]

For $\gamma>0$, the proximal mapping of $f$ can be computed coordinatewise in closed form:
\begin{equation}
\label{eq:num-prox}
\text{prox}_{\gamma f}(x)
=
\begin{pmatrix}
\mathcal{S}_{\frac{\gamma}{1+\gamma}}\!\left(\frac{x_1}{1+\gamma}\right)\\[1ex]
\mathcal{S}_{\frac{\gamma}{1+1000\gamma}}\!\left(\frac{x_2}{1+1000\gamma}\right)
\end{pmatrix},
\end{equation}
where $\mathcal{S}_\tau(\xi)=\mathrm{sign}(\xi)\max\{|\xi|-\tau,0\}$ denotes the soft-thresholding operator. Accordingly, the gradient of the Moreau envelope is given by
\begin{equation}
\label{eq:num-grad-moreau}
\nabla f_\gamma(x)=\frac{1}{\gamma}\bigl(x-\text{prox}_{\gamma f}(x)\bigr).
\end{equation}

All differential systems are solved numerically by \texttt{ode45} in MATLAB. The absolute and relative tolerances are chosen as
\[
\texttt{AbsTol}=10^{-10},
\qquad
\texttt{RelTol}=10^{-8}.
\]
The integration interval is $[t_0,T]=[1,50]$, and the initial data are
\[
x(t_0)=x_0=(20,-15),
\qquad
\dot x(t_0)=v_0=(0,0).
\]

For each trajectory $x(\cdot)$, we monitor the following quantities:
\begin{align}
\label{eq:num-gap}
f\bigl(\text{prox}_{\gamma(t)f}(x(t))\bigr)-f^\star,
\qquad
\|\nabla f_{\gamma(t)}(x(t))\|,
\qquad
\|x(t)\|.
\end{align}
We also plot the trajectory in the phase plane $(x_1(t),x_2(t))$.

\medskip

\noindent
\textbf{Remark on the diagnostics.}
When comparing models that are run under different theory-prescribed smoothing schedules, the raw values of
\[
f(\text{prox}_{\gamma(t)f}(x(t)))-f^\star
\quad\text{and}\quad
\|\nabla f_{\gamma(t)}(x(t))\|
\]
may start at very different levels for reasons that are intrinsic to the smoothing scale rather than to the quality of the dynamics. For this reason, in the cross-model comparison we also use the relative quantities
\begin{equation}
\label{eq:relative-gap}
\frac{f(\text{prox}_{\gamma(t)f}(x(t)))-f^\star}
     {f(\text{prox}_{\gamma(t_0)f}(x(t_0)))-f^\star},
\qquad
\frac{\|\nabla f_{\gamma(t)}(x(t))\|}
     {\|\nabla f_{\gamma(t_0)}(x(t_0))\|},
\end{equation}
which provide a scale-free comparison of the decay produced by each model.

\subsection{Influence of the parameter $\beta$}

We first study the influence of the parameter $\beta$ in our proposed dynamic. In this experiment, $\alpha$ is fixed at
\[
\alpha=4,
\]
while $\beta$ is varied over the set
\[
\beta\in\{0.01, 0.08, 0.8, 1.5\}.
\]
The functions $\delta(t)$ and $\gamma(t)$ are chosen according to
\[
p=0.5 \in (0, \alpha-3),
\qquad
\delta(t)=t^p=t^{0.5},
\qquad
\gamma(t)=0.01\,t^{p+2}=0.01t^{2.5}.
\]
These choices satisfy the polynomial setting in Remark~\ref{remark 3.1}. 

The corresponding numerical results are displayed in Figure~\ref{fig:beta}. The objective-value and gradient plots show that all choices of \(\beta\) lead to the same qualitative asymptotic behavior, in agreement with the fact that the convergence rates given by Theorem~\ref{first convergence result} are governed by the factor \(t^2\delta(t)\), hence here by \(t^{2.5}\), and do not explicitly depend on \(\beta\). More precisely, for the present scaling one has
\[
f\bigl(\prox_{\gamma(t)f}(x(t))\bigr)-f^\star=o(t^{-2.5}),
\qquad
\|\nabla f_{\gamma(t)}(x(t))\|=o(t^{-2.5}),
\]
as \(t\to+\infty\). 

The role of \(\beta\) is instead clearly visible in the transient behavior and in the geometry of the trajectories. The term weighted by \(\beta\) is a Hessian-driven damping term, and such terms are known to attenuate the oscillations of inertial systems. This effect is precisely what is observed in Figure~\ref{fig:beta}: as \(\beta\) increases, the oscillations in the objective, gradient and norm of trajectory curves become less pronounced, and the trajectories in the phase plane approach the minimizer in a more regular and less oscillatory way.

\begin{figure}[!ht]
	\centering
    \subfloat[Objective gaps]{\includegraphics[width = 2.9in]{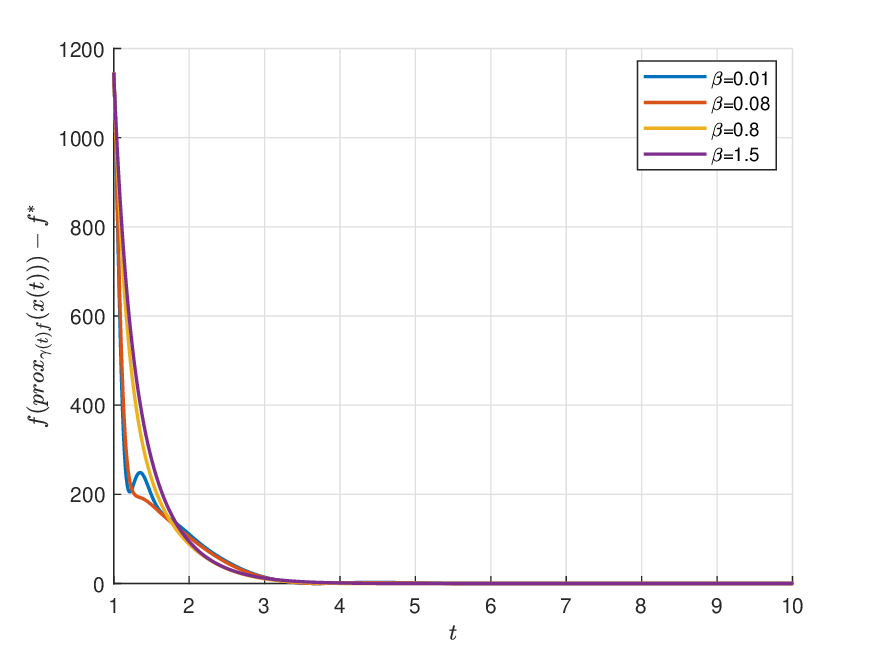}} 
	\subfloat[Gradient norms]{\includegraphics[width = 2.9in]{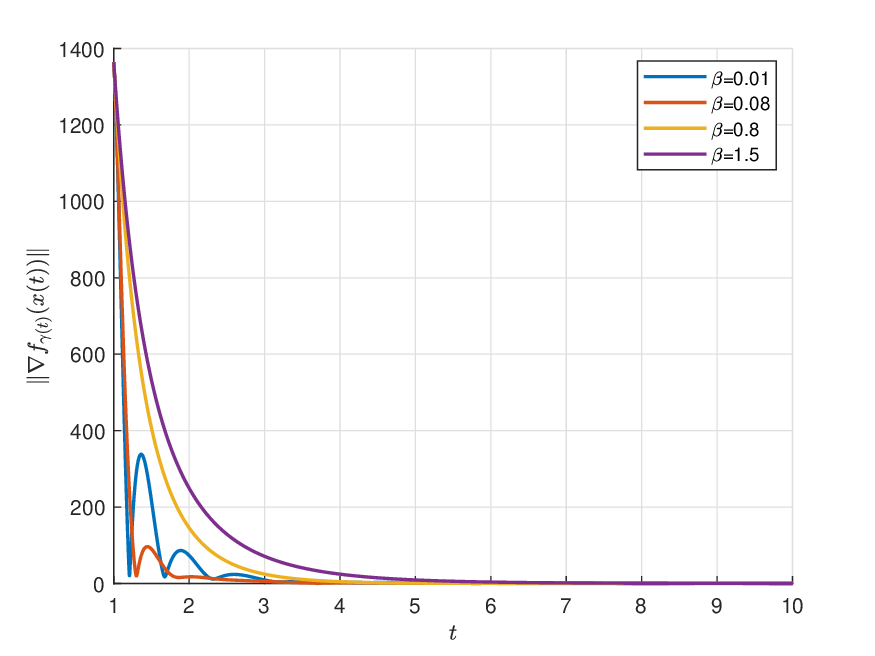}} \\
    \subfloat[Norms of the trajectories]{\includegraphics[width = 2.9in]{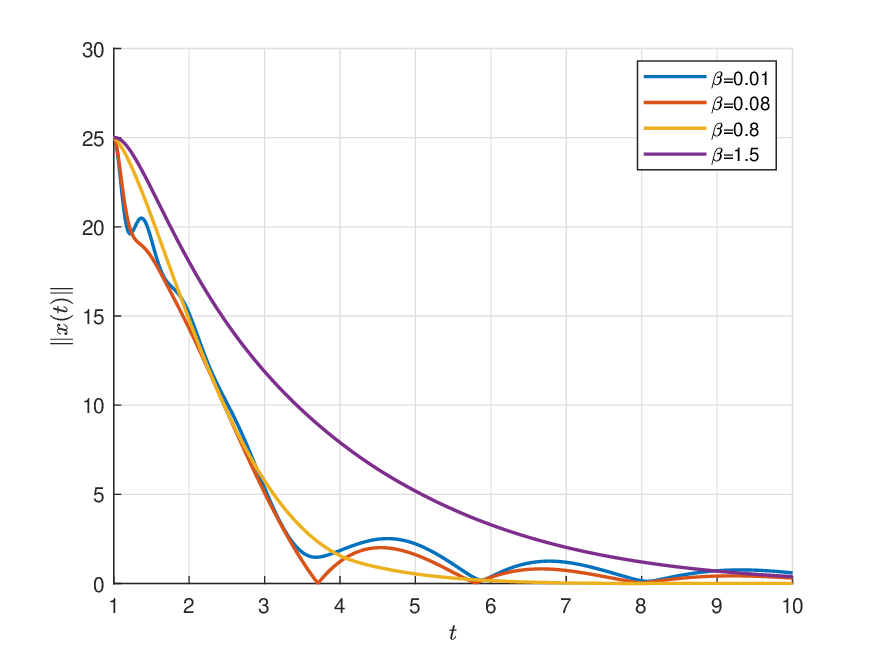}} 
	\subfloat[Phase-plane trajectories]{\includegraphics[width = 2.9in]{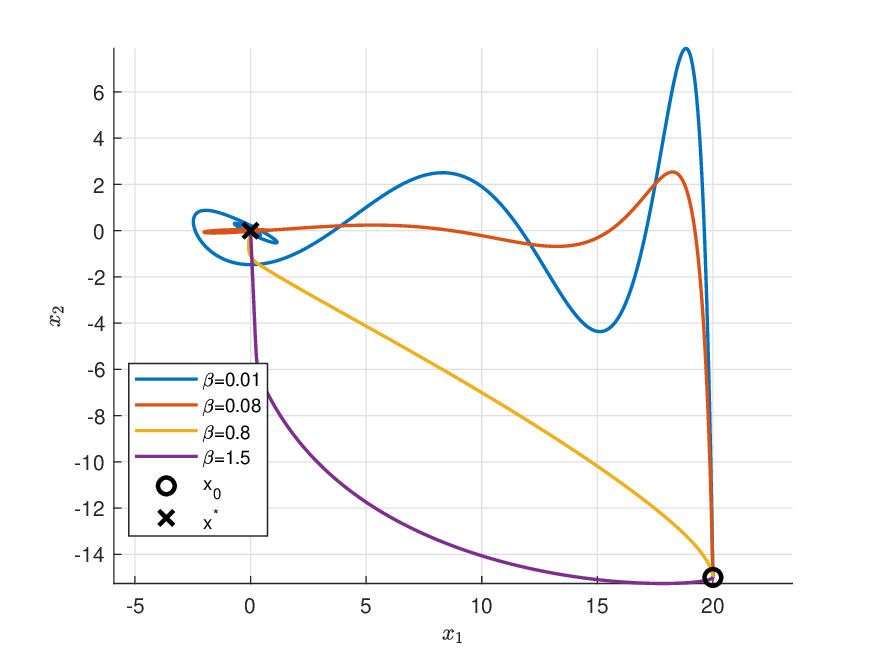}}
	\caption{Influence of the parameter $\beta$ on the behavior of \((\mathrm{NS\text{-}HR})_{\alpha,\beta}\).}
    \label{fig:beta}
\end{figure}

\subsection{Influence of the parameter $\alpha$}

We next analyze the effect of the parameter $\alpha$ while keeping $\beta$ fixed at
\[
\beta=1.
\]
The values of $\alpha$ are taken in the set
\[
\alpha\in\{4, 5, 6, 8\}.
\]
The functions $\delta(t)$ and $\gamma(t)$ are chosen according to
\[
p=\alpha-3.5 \in (0, \alpha-3),
\qquad
\delta(t)=t^p=t^{\alpha-3.5},
\qquad
\gamma(t)=0.01\,t^{p+2}=0.01\,t^{\alpha-1.5}.
\]
Again, these schedules fall within the polynomial regime covered by Remark~\ref{remark 3.1}. 

In this regime, Theorem~\ref{first convergence result} yields
\[
f\bigl(\prox_{\gamma(t)f}(x(t))\bigr)-f^\star
=o\!\left(\frac1{t^2\delta(t)}\right)
=o\!\left(t^{-(\alpha-1.5)}\right),
\]
and
\[
\|\nabla f_{\gamma(t)}(x(t))\|
=o\!\left(\frac1{t^2\delta(t)}\right)
=o\!\left(t^{-(\alpha-1.5)}\right).
\]
Therefore, the theoretical prediction is that larger values of \(\alpha\) should lead to faster convergence, because the exponent \(\alpha-1.5\) increases with \(\alpha\). 
This trend is clearly reflected in Figure~\ref{fig:alpha}. In both the objective and gradient plots, larger values of \(\alpha\) produce faster decay, and the separation between the curves is consistent with the ordering suggested by the theoretical rates. 

\begin{figure}[!ht]
	\centering
    \subfloat[Objective gaps]{\includegraphics[width = 2.9in]{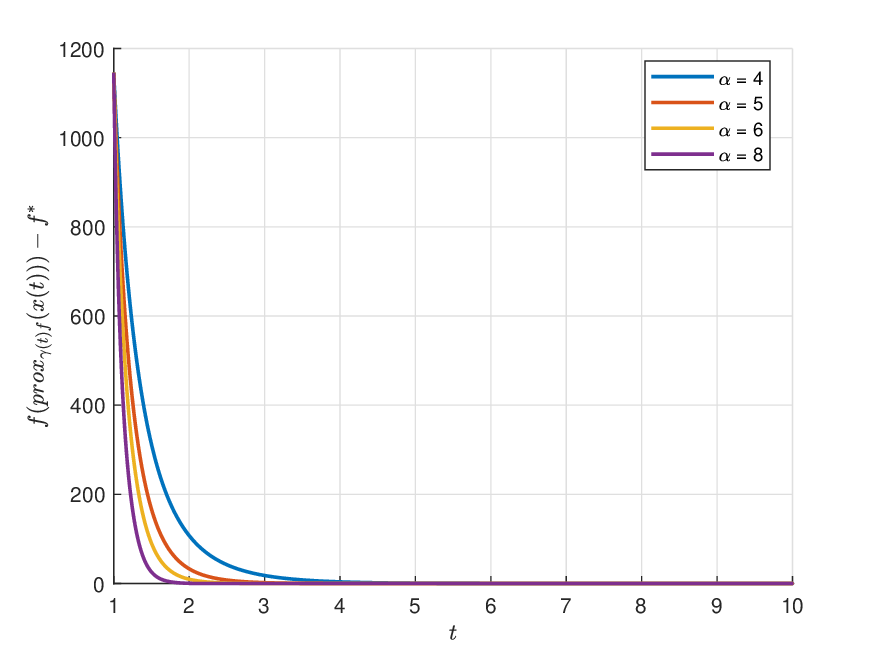}} 
	\subfloat[Gradient norms]{\includegraphics[width = 2.9in]{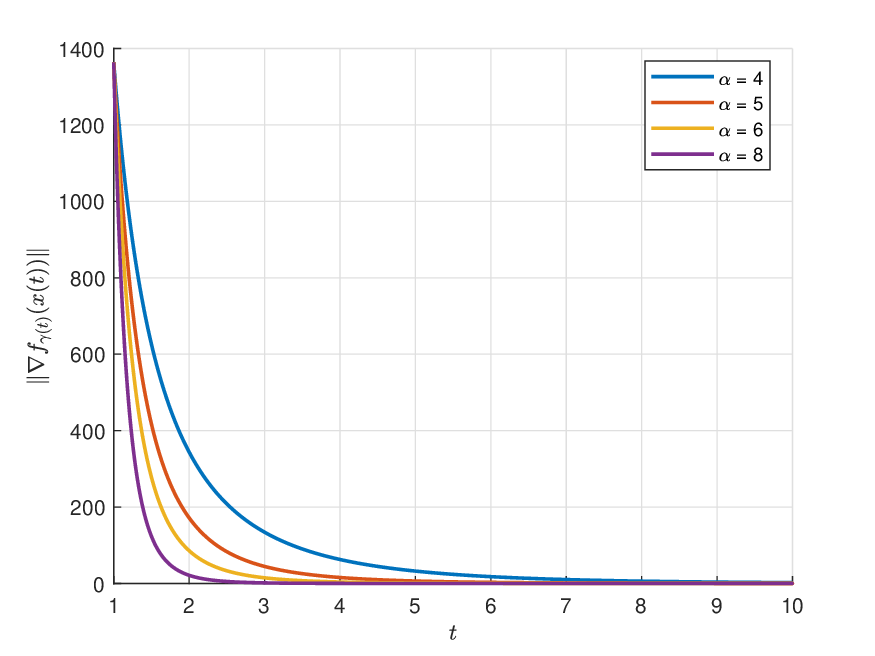}}
	\caption{Influence of the parameter $\alpha$ on the behavior of \((\mathrm{NS\text{-}HR})_{\alpha,\beta}\).}
    \label{fig:alpha}
\end{figure}

\subsection{Comparison with baseline and reference dynamics}

We now compare our proposed dynamic with four benchmark models:
\begin{enumerate}
    \item the reduced system obtained by setting $\beta=0$ while keeping the same factor $\delta(t)$,
    \item the further simplified system obtained by setting $\beta=0$ and $\delta(t)\equiv 1$,
    \item the Attouch--L\'aszl\'o Newton-like inertial dynamic ~\cite{AttouchLaszlo2021},
    \item the Bot--Karapetyants time-scaled Newton-like dynamic ~\cite{BotKarapetyants2022}.
\end{enumerate}

\paragraph{Our dynamic.}
For our model, the parameters are chosen as
\[
\alpha=4,
\qquad
\beta=1,
\qquad
\delta(t)=t^{0.5},
\qquad
\gamma(t)=10^{-4}t^{2.5}.
\]
According to Theorem~\ref{first convergence result}, this choice yields
\[
f\bigl(\prox_{\gamma(t)f}(x(t))\bigr)-f^\star=o(t^{-2.5}),
\qquad
\|\nabla f_{\gamma(t)}(x(t))\|=o(t^{-2.5}),
\]
and weak convergence of \(x(t)\) to a minimizer. Our dynamic achieves, for this choice of parameters, the accelerated objective decay \(o(t^{-2.5})\) while also preserving the same \(o(t^{-2.5})\) rate at the level of the Moreau-envelope gradient. 

\paragraph{Baseline dynamics.}
The two baseline models are run with the same values of $\alpha$, the same initial conditions, and the same smoothing schedule $\gamma(t)$ as our dynamic.

\paragraph{Attouch--L\'aszl\'o dynamic.}
The Attouch--L\'aszl\'o model is the Newton-like inertial dynamic
\begin{equation}
\label{eq:attouch-numexp}
\ddot x(t)+\frac{\alpha}{t}\dot x(t)
+\beta \frac{d}{dt}\nabla f_{\lambda(t)}(x(t))
+\nabla f_{\lambda(t)}(x(t))=0.
\end{equation}
For the Attouch--L\'aszl\'o model, we use the parameter regime
\[
\alpha=4,
\qquad
\beta=1,
\qquad
\lambda(t)=\lambda t^2,
\qquad
\lambda=\frac{1.1}{9}.
\]
This choice satisfies the assumptions of their convergence theorem and their theoretical results implies the weak convergence of the trajectory together with
\[
f(\text{prox}_{\lambda(t)f}(x(t)))-f^\star=o(t^{-2}),
\qquad
\|\nabla f_{\lambda(t)}(x(t))\|=o(t^{-2}).
\]
It is worth emphasizing that this \(o(t^{-2})\) objective rate is the best theoretical rate obtained in their paper for the considered model. 
\paragraph{Bot--Karapetyants dynamic.}
The Bot--Karapetyants model is the time-scaled Newton-like dynamic
\begin{equation}
\label{eq:bot-numexp}
\ddot x(t)+\frac{\alpha}{t}\dot x(t)
+\beta(t)\frac{d}{dt}\nabla f_{\lambda(t)}(x(t))
+b(t)\nabla f_{\lambda(t)}(x(t))=0.
\end{equation}
Compared with \eqref{eq:attouch-numexp}, this system introduces an additional time-scaling function $b(t)$ in front of the regularized gradient, and a possibly time-dependent coefficient $\beta(t)$ in front of the Hessian term. \\
For the Bot--Karapetyants model, we choose
\[
\alpha=4,\qquad
\beta(t)=1,\qquad
b(t)=4.1\,t^{0.5},\qquad
\lambda(t)=t^{0.5},
\]
Their theoretical results implies weak convergence of the trajectory and the rates
\[
f(\text{prox}_{\lambda(t)f}(x(t)))-f^\star=o(t^{-2.5}),
\qquad
\|\nabla f_{\lambda(t)}(x(t))\|=o(t^{-1.5}).
\]
Thus, in the chosen regime, the Bot--Karapetyants dynamic has the same theoretical objective-value decay \(o(t^{-2.5})\) as our dynamic, but a slower gradient decay. More generally, both the Bot--Karapetyants model and the proposed dynamic can in principle achieve arbitrarily fast objective decay by a suitable choice of the time-dependent coefficients; here we deliberately selected parameters so that the two systems have the same theoretical objective-value rate, making the comparison more informative. 
\begin{figure}[!ht]
	\centering
    \subfloat[Objective gaps]{\includegraphics[width = 2.9in]{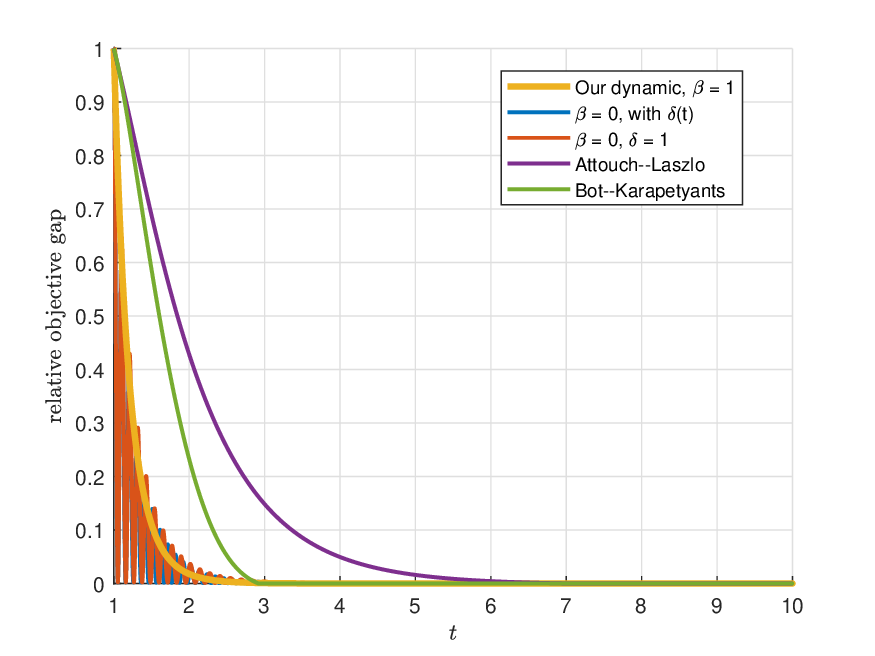}} 
	\subfloat[Gradient norms]{\includegraphics[width = 2.9in]{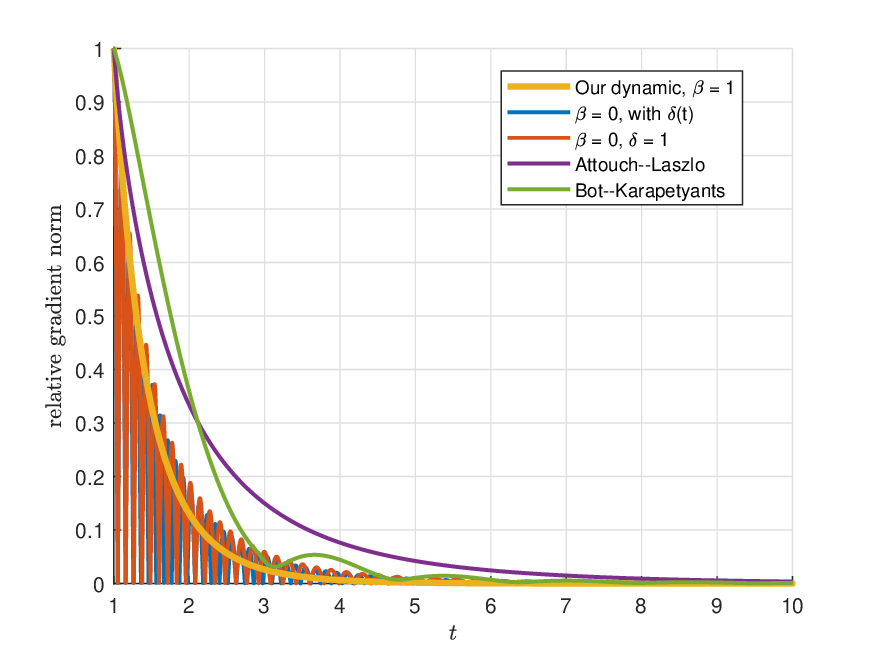}} \\
    \subfloat[Norms of the trajectories]{\includegraphics[width = 2.9in]{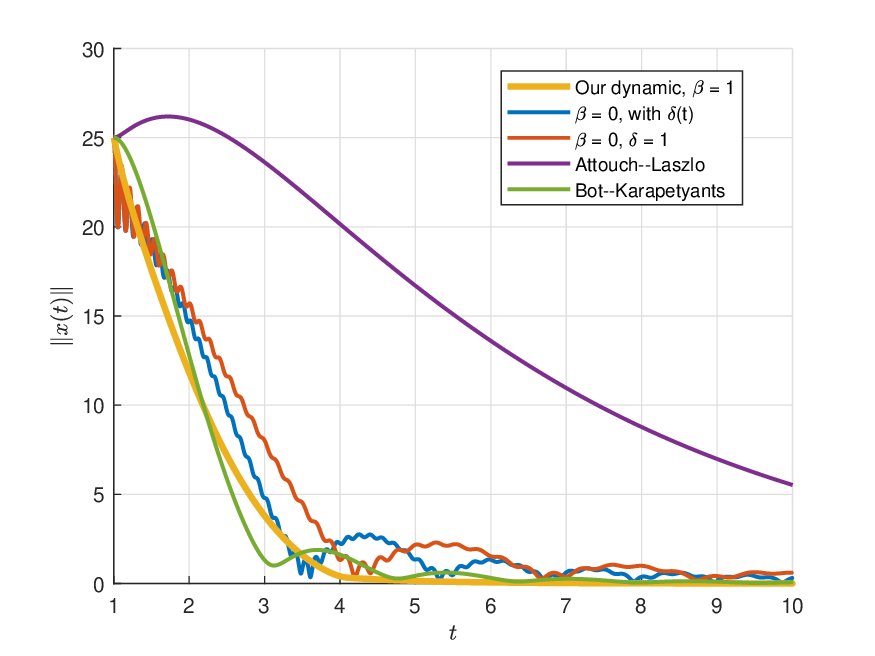}} 
	\subfloat[Phase-plane trajectories]{\includegraphics[width = 2.9in]{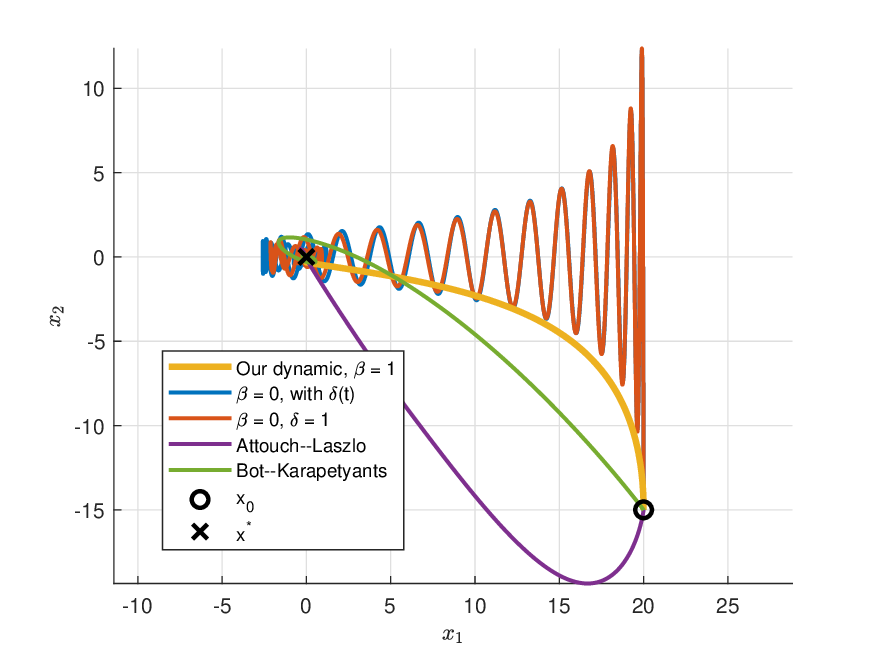}}
	\caption{Comparison of \((\mathrm{NS\text{-}HR})_{\alpha,\beta}\) with four benchmark models: the two baseline systems, the
Attouch--L\'aszl\'o dynamic, and the Bot–Karapetyants dynamic.}
    \label{fig:comparison}
\end{figure}
The comparison plots in Figure~\ref{fig:comparison} show that the proposed dynamic compares favorably with all four benchmark systems. Relative to the two baseline dynamics, the proposed model exhibits a clearly more regular behavior, with reduced oscillations and faster decay of both the normalized objective and the normalized gradient. This confirms numerically that the combination of Hessian-driven damping and time-rescaled regularized gradients improves both stability and convergence speed.

Compared with the Attouch--L\'aszl\'o dynamic, our method also performs better in the present experiment. This is consistent with the fact that our chosen scaling yields the faster theoretical objective and gradient rates \(o(t^{-2.5})\), whereas the Attouch--L\'aszl\'o model is limited here to the rate \(o(t^{-2})\).

The most interesting comparison is with the Bot--Karapetyants dynamic. By construction, the two systems have the same theoretical rate \(o(t^{-2.5})\) for the objective values. The numerical results nevertheless show an advantage for the proposed model, which can be seen in the plots in Figure~\ref{fig:comparison}. In particular, not only does our dynamic perform better in function value, it also displays a visibly stronger stabilization and a faster decrease of the Moreau-envelope gradient, in line with the theoretical prediction \(o(t^{-2.5})\) for our model versus \(o(t^{-1.5})\) for the Bot--Karapetyants dynamic.

\subsection{Discussion}

The numerical experiments support the theoretical analysis developed in the paper. First, the parameter study with respect to \(\beta\) shows that the Hessian-driven damping term mainly affects the transient regime by attenuating oscillations and stabilizing the trajectories, while preserving the same asymptotic decay order. Second, the study with respect to \(\alpha\) is fully consistent with Theorem~\ref{first convergence result}: larger values of \(\alpha\) lead to larger values of the exponent \(\alpha-1.5\), and the corresponding plots indeed display faster convergence.

Finally, the comparison with the benchmark dynamics highlights the main strength of the proposed model. The new dynamic combines the damping effect of high-resolution Newton-like systems with the acceleration induced by time-rescaled gradients. In the present numerical tests, this leads to smaller oscillations, faster practical convergence, and, in the matched-rate comparison with the Bot--Karapetyants system, a better balance between objective decrease and gradient decay. Altogether, these results illustrate the practical relevance of the proposed continuous-time model and support its interpretation as a nonsmooth high-resolution inertial dynamic.

\section{Conclusion}

We introduced a nonsmooth high-resolution inertial dynamic combining vanishing viscous damping, Hessian-driven damping, and time-rescaled gradients through the Moreau envelope. After establishing an equivalent first-order formulation, we proved global existence and uniqueness of trajectories and derived convergence properties including fast decay of the objective residual and the Moreau-envelope gradient, stabilization of velocities, and weak convergence to minimizers under suitable assumptions on the time-dependent parameters. We also extended the framework to maximally monotone operators via the Yosida approximation, showing that the proposed approach applies beyond convex minimization. The numerical experiments illustrate the influence of the parameters \(\beta\) and \(\alpha\), confirm the trends predicted by the theory, and show favorable behavior of the proposed system compared with several benchmark dynamics. Finally, we derived a natural proximal discretization of the continuous model, whose full Lyapunov and convergence analysis is left for future work; it would also be interesting to explore extensions of our dynamics to composite, constrained, or stochastic problems. Overall, the results of this paper indicate that the proposed dynamic provides a flexible and effective nonsmooth high-resolution framework, combining acceleration, damping, and time rescaling in a unified manner.

\section*{Acknowledgements}

This work was partly supported by the Agence Nationale de la Recherche (ANR) with the project ANR-23-CE48-0011-01.

\end{document}